\documentclass[reqno,11pt]{amsart}
\usepackage{geometry}
\usepackage[utf8]{inputenc}
\usepackage{fullpage}
\usepackage{xifthen}
\usepackage{amsmath,amsthm,amssymb}
\usepackage{amsfonts,bm,version}
\usepackage{graphicx,fancybox,mathrsfs}
\usepackage{color,xcolor}
\usepackage{mathtools}
\usepackage{enumerate}
\usepackage{caption}
\usepackage[pagewise]{lineno}
\usepackage{tikz}
\usetikzlibrary{positioning} 
\usepackage{xcolor} 
\usepackage[colorlinks,linkcolor=blue,anchorcolor=green,citecolor=red,]{hyperref}
\allowdisplaybreaks

\definecolor{bleu1}{RGB}{0,57,128}
\def\bleu1{\color{bleu1}}

\usepackage{etoolbox}
\patchcmd{\section}{\normalfont}{\normalfont \color{violet}}{}{}
\patchcmd{\subsection}{\normalfont}{\normalfont \bleu1}{}{}
\patchcmd{\subsubsection}{\normalfont}{\normalfont \bleu1}{}{}

\newtheorem{definition}{Definition}[section]

\newtheorem{theorem}{Theorem}[section]
\newtheorem{lemma}[theorem]{Lemma}
\newtheorem{prop}[theorem]{Proposition}
\newtheorem{corollary}[theorem]{Corollary}
\theoremstyle{remark}
\newtheorem{remark}{Remark}[section]

\numberwithin{equation}{section}

\newenvironment{proof1}[1][\textbf{Proof of Lemma \ref{lemma.basic}.}]{\noindent\textit{#1}\quad } {\hfill$\Box$\vspace{0.7mm}}

\newenvironment{proof2}[1][\textbf{Proof of Lemma \ref{lemma.hderivative}.}]{\noindent\textit{#1}\quad } {\hfill$\Box$\vspace{0.7mm}}

\newenvironment{proof3}[1][\textbf{Proof of Lemma \ref{lemma.Ihyperbolic}.}]{\noindent\textit{#1}\quad } {\hfill$\Box$\vspace{0.7mm}}

\def \beq{\begin{equation}}
\def \eeq{\end{equation}}
\def \beqs{\begin{equation*}}
\def \eeqs{\end{equation*}}
\def \bea{\begin{eqnarray}}
\def \eea{\end{eqnarray}}
\def \beas{\begin{eqnarray*}}
\def \eeas{\end{eqnarray*}}
\def \brk{\begin{remark}}
\def \erk{\end{remark}}
\def \bdf{\begin{definition}}
\def \edf{\end{definition}}
\def \bthm{\begin{theorem}}
\def \ethm{\end{theorem}}
\def \bl{\begin{lemma}}
\def \el{\end{lemma}}

\graphicspath{{./figures/}}

\begin{document}

\begin{abstract}
	We construct examples of discontinuity  of Lyapunov exponent in the spaces of quasiperiodic $\mathrm{SL}(2,\mathbb R)$-cocycles for fixed irrational frequencies. Especially, we prove that the Gevrey space $G^2$  is the transition space of continuity for  all strong Diophantine frequencies. We also construct examples of discontinuity for other frequencies in less smooth spaces, which  show that the more difficult it is to approximate the frequency with rational numbers, the more likely it is to exhibit discontinuity in smoother spaces.

\

\end{abstract}

\title{Discontinuity of Lyapunov exponent in spaces of quasiperiodic cocycles: Smoothness vs Arithmetic}
\author{Jinhao Liang, Kai Tao, and Jiangong You}
\address{School of Mathematics, Hohai University, Nanjing 210098, China}\email{jhliang$\underline{\ }$hhu@163.com}
\address{School of Mathematics, Hohai University, Nanjing 210098, China}\email{ktao@hhu.edu.cn}
\address{Chern Institute of Mathematics and LPMC, Nankai University, Tianjin 300071, China} \email{jyou@nankai.edu.cn}

\maketitle

\section{Introduction.}

Let $X$ be a $C^l$ compact manifold, $A(x)$ be a $\mathrm{SL}(2,\mathbb R)$-valued function on $X$ and $(X,T,\mu)$ be ergodic with $\mu$ a normalized $T$-invariant measure. Dynamical systems on $X\times \mathbb R^2$ given by
$$(x,w)\mapsto (Tx,A(x)w)$$
are called a $\mathrm{SL}(2,\mathbb R)$-cocycle over $(X,T,\mu)$ and denoted by $(T,A)$.
We call the special  case  $X=\mathbb S^1$, $\mu$ being Lebesgue measure on $\mathbb S^1$ and $T=T_\alpha:x\mapsto x+\alpha$ with $\alpha$ irrational,   quasiperiodic $\mathrm{SL}(2,\mathbb R)$-cocycles, denoted by $(\alpha, A)$ for simplicity.  If  $A(x)=S_{E,v}(x):=\left(
\begin{array}{cc}
	E-v(x) & -1\\
	1 & 0 \\
\end{array}
\right)
$, we call them  Schr\"odinger cocycles.

The $n$th iteration of the cocycle is given by
$$(\alpha,A)^n=(n\alpha, A^n),$$
where
$$
A^n(x)=\left\{
                         \begin{array}{ll}
                           A(T^{n-1}x)\cdots A(x), & \hbox{ $n\geq1$;} \\
                           Id_2, & \hbox{ $n=0$;} \\
                           {\left(A^{-n}(T^nx)\right)}^{-1}, & \hbox{ $n\leq-1$.}
                         \end{array}
                       \right.
$$
The average exponential growth of the norm of $A^n(x)$ is measured by the Lyapunov exponent (LE for short)
\beqs
L(\alpha,A)= \lim_{n\rightarrow\infty}\frac 1n \int_{\mathbb S^1} \log\|A^n(x)\|dx.
\eeqs
The existence and non-negativity of $ L(\alpha,A)$ are guaranteed by subadditivity of the sequence $\{\int_{\mathbb S^1}\log\|A^n(x)\|dx\}_{n\geq 1}$ and  the fact that $A^n(x)\in \mathrm{SL}(2,\mathbb R)$.
Moreover, by Kingman's subadditive ergodic theorem
\beqs
L(\alpha,A)= \lim_{n\rightarrow\infty}\frac 1n \log\|A^n(x)\|,
\eeqs
for $\mu$-almost every $x\in \mathbb S^1$.

   In Bourgain-Jitomirskaya's seminar paper \cite{BJ02}, they prove that the
   LE is jointly continuous in $(\alpha,E)\in \mathbb R\backslash \mathbb Q\times \mathbb R$
   for Schr\"odinger cocycles $S_{E,v}(x)$ when $v(x) $ is analytic and fixed. It is a cornerstone of the analytical theory of quasiperiodic Schr\"odinger operators, and many significant results, including the solution to the Ten Martini Problem and Avila's global theory, have been established based on it.
    Later, the continuity in $(\alpha, A)\in {\mathbb R}\backslash {\mathbb Q}\times C^\omega({\mathbb S}^1, \mathrm{SL}(2, {\mathbb R}))$ was proved in \cite{JKS,JM11,JM12}.

  For a long time, people expected  the continuity could be extended to larger spaces,
 such as the Gevrey spaces defined by $G^{\, s}(\mathbb S^1,\mathrm{SL}(2,\mathbb R)) =\{A(x):  \  \|A^{\, (k)}(x)\| < C (k!)^s R^k, C>0, R>0\}$ where $A^{\, (k)}(x)$ denotes the $k$-order derivative of $A(x)$,  which lies in between $C^\infty$ and $C^\omega$, even  more larger spaces $C^k,  k=1,2,\cdots, \infty$.
  By approximation method, the continuity of  $L(\alpha, A)$ was succeeded to be extended to the Gevrey space $G^s, 1<s<2$ for Diophantine frequencies \cite{CGYZ,K}. It is also well known that $L(\alpha, S_{E,v})$ could be continuous even for some $C^2$ potentials \cite{WZ15}.
 However, $L(\alpha, A)$ has been proved to be discontinuous in $A\in  C^0({\mathbb S}^1, \mathrm{SL}(2, {\mathbb R}))$ by Furman \cite{F} and Bochi \cite{Bo99,Bo02} for any irrational frequency $\alpha$ at any nonuniformly hyperbolic cocycle, which implies that  $C^0({\mathbb S}^1, \mathrm{SL}(2,{\mathbb R}))$  is too large for the continuity of LEs.
 So, identifying the largest space for continuity of LE, and thus identifying the boundary of the ``analytical theory" is certainly an important issue. People in the community once expected that continuity holds in $C^k({\mathbb S}^1, \mathrm{SL}(2, {\mathbb R}))$ for sufficiently large $k$ (similar to KAM theory). Unfortunately, it is not true due to counter-examples in $C^\infty({\mathbb S}^1, \mathrm{SL}(2, {\mathbb R}))$ constructed by Wang-You \cite{WY13, WY18}. Later, it was proved that the Gevrey space $G^s, s=2$ is the transition space from continuity to discontinuity \cite{CGYZ, GWYZ, K}. Both Wang-You \cite{WY13,WY18} and Ge-Wang-You-Zhao \cite{GWYZ} assumed that the frequency  $\alpha$ is bounded type (see Definition \ref{1.2}), which is technically crucial in their construction of counterexamples. However, it is well known that the set of bounded type frequencies is of zero Lebesgue measure. So people in the community were curious if the continuity can cross $G^2$ for most frequencies as mentioned in Jitomirskaya's ICM plenary talk \cite{J}: ``This still leaves open the question whether continuous behavior of the Lyapunov exponents at least for Schr\"odinger cocycles with regularity lower than $G^2$ is possible if the frequency is not of bounded type."

 In this paper, we show that $G^2$ is in fact the transition space for almost all frequencies (more precisely for all strong Diophantine frequencies), and we also  construct counter-examples for other frequencies in larger spaces.
 Before stating our results, we first present the following classification on frequencies.


\begin{definition}\label{1.2}
Let $\{p_n/q_n\}$ be the fractional expansion of $\alpha$. We say that
	\begin{itemize}
		\item  $\alpha$ is bounded type if there exists $M>0$ such that \beqs q_{n+1}\leq M q_n ,\quad  \forall n\in \mathbb N.\eeqs
		\item  $\alpha$ is  strong Diophantine  if there exist $\gamma>0$ and $\tau>0$ such that \beq\label{sdiophantine} q_{n+1}\leq \gamma^{-1} q_n (\log q_n)^\tau,\quad  \forall n\in \mathbb N.\eeq
		\item   $\alpha$ is  Diophantine  if there exist $\gamma>0$ and $\tau>1$ such that
		\beq\label{diophantine}q_{n+1}\leq \gamma^{-1} q_n^\tau,\quad  \forall n\in \mathbb N.\eeq
		
		\item  $\alpha$ is  Brjuno  if there exists  $0<\delta<1$ such that \beqs \beta_\delta(\alpha):=\limsup_{n\rightarrow\infty} \frac{\log q_{n+1}}{q_n^\delta}<\infty.\eeqs

		\item $\alpha $ is  finite-Liouvillean  if \beqs \beta(\alpha) :=\limsup_{n\rightarrow\infty} \frac{\log q_{n+1}}{q_n}<\infty.\eeqs
		\item $\alpha$ is called super-Liouvillean if $\beta(\alpha)=\infty$.
	\end{itemize}
\end{definition}

Denote by $\mathrm{SDC}(\gamma,\tau)$ the set of all strong Diophantine numbers with (\ref{sdiophantine}), and by $\mathrm{SDC}=\bigcup_{\gamma>0,\tau>0} \mathrm{SDC}(\gamma,\tau)$.
Denote by $\mathrm{DC}(\gamma,\tau)$ the set of all Diophantine numbers with (\ref{diophantine}), and by $\mathrm{DC}(\tau)=\cup_{\gamma>0}\mathrm{DC}(\gamma,\tau)$.
It is known that $\mathrm{SDC}$ and $\mathrm{DC}(\tau)$ are sets of full Lebesgue measure for any $\tau>1$.

\begin{theorem}\label{thm.main5}
For any fixed $\alpha \in \mathrm{SDC}$, the Gevrey space $G^2$ is the transition space of the continuity of LE.
That is, the Lyapunov exponent $L(\alpha, A)$ is continuous in $G^s(\mathbb S^1,\mathrm{SL}(2,\mathbb R))$  for $1<s<2$,  while  it has discontinuous points in  $ G^s(\mathbb S^1, \mathrm{SL}(2,\mathbb R))$  for  $s>2$.
\end{theorem}

Based on Theorem \ref{thm.main5}, we can obtain the transition space on potentials for quasiperiodic Schr\"odinger cocycles by same argument as in \cite{WY13} (See page 2367, proof of Theorem 2 in \cite{WY13}).

\begin{corollary}
	For any fixed $\alpha\in \mathrm{SDC}$ and $E\in\mathbb{R}$, $L(\alpha,S_{E,v})$ is continuous with respect to $v$ in $G^s(\mathbb S^1,\mathbb R)$ if $1<s<2$, while there exists $v_0\in G^s(\mathbb S^1)$ such that $L(\alpha,S_{E,v})$ is discontinuous at $v_0$ in $G^s(\mathbb S^1,\mathbb R)$ if $s>2$.
\end{corollary}

We also give counterexamples for all irrational frequencies except super-Liouvillean one.

\begin{theorem}\label{thm.main3}
 For any fixed Diophantine frequency $\alpha\in \mathrm{DC}(\tau)$, there exists $D\in G^s(\mathbb S^1, \mathrm{SL}(2,\mathbb R))$ such that the LE is discontinuous at $D$ in $G^s$ topology with $s>\tau +1$.
\end{theorem}

\begin{theorem}\label{thm.main2}
 Fixed any fixed Brjuno frequency, there exists $D\in C^\infty(\mathbb S^1, \mathrm{SL}(2,\mathbb R))$ such that the LE is discontinuous at $D$ in $C^\infty$ topology.

\end{theorem}

\begin{theorem}\label{thm.main1}
	 Fixed any fixed finite-Liouvillean frequency, there exists $D\in C^\ell(\mathbb S^1, \mathrm{SL}(2,\mathbb R))$ such that the LE is discontinuous at $D$ in $C^\ell$ topology for any $\ell\in \mathbb Z_+$.
\end{theorem}

We now outline the central idea of our construction. The LE is a measure of the limiting exponential growth rate of matrix production. To grasp this concept, consider the norm of the product of two $\mathrm{SL}(2,\mathbb{R})$-matrices.  For two matrices $A$ and $B$ with large norms, the norm of their product $AB$ is determined by the norms of $A$ and $B$, as well as the angle between the expanding and contracting directions of $A$ and $B$ (see Lemmas \ref{lemma.basic} and \ref{lemma.resonantnorm}). In essence, the norm of $AB$ increases when the angle is relatively large (non-resonant case), and conversely, it has cancelation  when the angle is relatively small (resonant case). For a finite set of matrices $\{A_1, \cdots, A_n\}$ with large norms, approximately $\lambda \gg 1$, the norm of the product $A_1 \cdots A_n$ is roughly $O(\lambda^n)$, provided that the angles between adjacent matrices are not excessively small.

Our strategy involves inductively constructing a convergent sequence of cocycles $\{A_n, n = N, N+1, \cdots\}$, each of the form $\Lambda R_\phi(x)$, where $\Lambda = \text{diag}\{\lambda, \lambda^{-1}\}$ and $\phi(x)$ possesses a unique degenerate zero, we call it the critical point, say at $x = 0$.
Let $I_n$ denote a decreasing, small neighborhood around $x = 0$. Define the increasing function $r_n$ as the minimal first returning time of $I_n$ under the transformation $T_\alpha: x \mapsto x + \alpha$. We observe that $r_n$ is dependent on both $I_n$ and the fractional expansion of $\alpha$. Since $\phi(x + \alpha), \ldots, \phi(x + (r_n - 1)\alpha)$ are not very close to zero, as a result we have $\|A(x + (r_n - 1)\alpha) \cdots A(x)\| = O(\lambda^{(1-\epsilon)r_n})$. Consequently, the finite LE of $(\alpha, A_n)$ up to order $r_n$ is $(1-\epsilon)\log\lambda$ with $\epsilon$ arbitrarily small. This implies that the LE of the limiting cocycle $(\alpha, D)$ is roughly $(1-\epsilon)\log\lambda$ which is obviously  positive if $\lambda$ is big although  we do not yet know whether the LE of $(\alpha, A_n)$ is positive or not\..

Meanwhile, since all $A_n$ possess a degenerate critical point, the degeneracy dictates the smoothing class in which the counterexample will be constructed. This degeneracy enables us to perturb $A_n$ to $\tilde{A}_n$ in the vicinity of $x=0$ in desired topological space. Such perturbations ensure that $\tilde A_n\to D$ in the smoothing class in which the counterexample will be constructed and  all $L(\tilde{A}_n)$ are significantly smaller than $L(D)$, which in turn causes discontinuity.

 The construction of $\{A_{n}\}$ is given by an induction, more precisely, the construction of $A_{n+1}$ is based on the product of $A_n$ along much longer orbits of $T_\alpha$.
  However, since these orbits are bound to enter $I_n$ multiple times, each entry can potentially lead to a cancellation. The occurrence of this cancellation is contingent upon how the orbits converge near $x=0$, which is influenced by the arithmetic properties of $\alpha$, specifically the growth rate of $q_n$, and the flatness of $\phi(x)$ in the neighborhood of $x=0$.
 To successfully construct the desired object, we must avoid these cancellations. Consequently, $\phi(x)$ cannot be overly flat when $\alpha$ is more closely approximated by rational numbers (we term this ``bad frequency''). This is the reason why counterexamples can only be constructed within a larger space for ``bad'' frequencies.

Although the main framework of the proof is similar to those found in \cite{GWYZ, WY13, WY18}, this paper introduces several innovations that allow us to extend the argument to its full potential. As a result, we achieve not only the optimal result for all strong Diophantine frequencies (not limited to bounded types) but also manage to construct counterexamples for other frequencies.

Now, we describe the innovations presented in this paper. The first key innovation is the design of critical intervals $I_n$ for each $\alpha$ based on the fractional expansion of $\alpha$. This design ensures that the returning times within these intervals can be accurately computed.
More precisely, each interval $I_n$ is divided into two subintervals. One subinterval has a returning time of $q_n$, while the other subinterval has a returning time of $q_{n+1}$ (as detailed in Lemma \ref{lemma.returntimeIn}). Consequently, the norms of the iterated cocycles are exactly given by $O(\lambda^{q_n})$ for the phases within the first subinterval and $O(\lambda^{q_{n+1}})$ for the phases in the other subinterval. This development allows us to generalize the results of \cite{GWYZ} from bounded types to all strong Diophantine frequencies.

The second innovation lies in the strategic selection of degeneracy at the critical point of the starting cocycle. This selection is optimized for economic efficiency. For constructing counterexamples within the $ C^l $ space, we employ a degeneracy order of $ |x|^{l+1} $ at the critical point. In the context of $ C^\infty $ space, we adopt a degeneracy of $ e^{-(-\log|x|)^\delta} $, where $ \delta $ is greater than 1. When dealing with $ G^s $, we utilize $ e^{-|x|^{-\frac{1}{s-1}}} $ as the degeneracy. This approach allows us to  construct counterexamples for other Diophantine and Liouvillean frequencies.

Our construction suggests that the process of constructing counterexamples within the domain of cocycles with smoother topology requires a ``more Diophantine'' frequency. This means that the closer the frequency is to be approximated by rational numbers, the more likely it is to display discontinuity in smoother spaces. This finding is somewhat counterintuitive, as many previous results have shown that Diophantine properties are beneficial for the H\"older continuity of LE \cite {GS01, HZ, LWY, YZ}. It would thus seem reasonable to expect that these properties also facilitate the continuity of LE.
However, our investigation in this paper let us guess that the Liouvilleaness of the frequency may actually aid in ensuring continuity, despite its known detrimental effects on H\"older continuity \cite{ALSZ}.
Our construction also hints that LE might be continuous in $C^\infty$ if the frequency $\alpha$ is very Liouvillean. To this stage, we expect positive answer to the following  problems.\\

\noindent
Problem 1. {\it Whether LE is continuous in $G^s$ topology with $1<s<2$ for all irrational frequencies?}\\

\noindent
Problem 2. {\it Whether LE is continuous in $G^s$ topology with $s<\tau+1$ for all Diophantine frequencies in $\mathrm{DC}(\tau)$?}\\

\noindent
Problem 3. {\it Whether LE is continuous in $C^k$ topology $(k=1, 2, \cdots)$ for super-Liouvillean frequencies?}\\

\noindent
Problem 4. {\it   LE as a function of $E$ in Schr\"odinger case is much more easy to be continuous when both the frequency and the  potential are fixed. We believe that LE is continuous for typical smooth potentials.  Both positive results besides and negative results are extremely interesting.  So far, only the big $C^2$ cosine type potential case has been proved \cite{WZ15}.}\\
\

The rest of the paper is organized as follows.
In Section 2, we introduce the nonresonant case, which is used in Young’s inductive argument \cite{Y} to prove positive LE, and the completely resonant case, which is the key to derive discontinuity of LE.
In Section 3, we obtain sharp distribution of the trajectory of the irrational shift and apply it to compute the returning times for single interval with suitable length.
In Section 4 and Section 5, we construct $\{A_n\}$ and $\{\tilde A_n\}$ in $C^l$ topology and then prove discontinuity of LE for finite-Liouvillean frequency.
We analyze the distribution of trajectory to obtain refined estimates on norms and then establish Young's inductive argument.
In Section 6 and Section 8, we follows the method in $C^l$ case to derive discontinuity in $C^\infty$ topology and Brjuno frequency.
Similarly, in Section 7 and Section 9, we prove discontinuity in $G^s$ topology and (strong) Diophantine frequency.

\

\section{Angles and norms of product of SL(2,$\mathbb R$)-matrices}
In this section, we analyze of angles and norms for SL(2,$\mathbb R$)-matrices and then introduce nonresonant case and completely resonant case.
In nonresonant case, the norms of a sequence of matrices are accumulated, which is the core of Young’s inductive argument.
In completely resonant case, the norm of two adjacent hyperbolic matrices mutually cancel out, which is the key to prove discontinuity.

Throughout the paper, we denote by $c$ and $C$ universal positive constants (the constants may change in different estimates). 

For $\theta\in \mathbb S^1$, let
$$R_\theta=\left(
             \begin{array}{cc}
               \cos\theta & -\sin\theta  \\
               \sin\theta & \cos\theta \\
             \end{array}
           \right).
$$
Define the map
$$s:\mathrm{SL}(2,\mathbb R)\rightarrow \mathbb{RP}^1=\mathbb R/(\pi \mathbb Z)$$
such that $s(A)$ is the most contracted direction of $A\in \mathrm{SL}(2,\mathbb R)$. That is, for any unit vector $\hat s(A)\in s(A)$, it holds that $\|A\cdot \hat s(A)\|=\|A\|^{-1}$. Abusing the notation a little, let
$$u:\text{SL}(2,\mathbb R)\rightarrow \mathbb{RP}^1=\mathbb R/(\pi \mathbb Z)$$
be determined by $u(A)=s(A^{-1})$ and $\hat u(A)\in u(A)$. Then by singular decomposition of $A\in \text{SL}(2,\mathbb R)\backslash \text{SO}(2,\mathbb R)$,
$$A=R_u\cdot \left(
             \begin{array}{cc}
               \cos\theta & -\sin\theta  \\
               \sin\theta & \cos\theta \\
             \end{array}
           \right)\cdot R_{\frac\pi2-s},$$
where $s,u\in [0,\pi)$ are angles corresponding to the directions $s(A),u(A)\in \mathbb R/(\pi \mathbb Z)$.

Now consider the product $E_3(x)=E_2(x)\cdot E_1(x)$, where $E_1,E_2\in C^1(\mathbb S^1,\text{SL}(2,\mathbb R)\backslash \text{SO}(2,\mathbb R)$) and $x\in \mathbb S^1$. Then we have
\beqs
E_3=E_2\cdot E_1=R_{u(E_2)}\left(\begin{array}{cc}\|E_2\| & 0 \\ 0 & \|E_2\|^{-1} \\
                       \end{array}
                \right)R_{\frac \pi 2-s(E_2)+u(E_1)}
                \left(\begin{array}{cc}\|E_1\| & 0 \\ 0 & \|E_1\|^{-1} \\
                      \end{array}
                \right)R_{\frac \pi 2-s(E_1)}.
\eeqs
We say $\left|s(E_2)-u(E_1)\right|$ is the angle between $E_2$ and $E_1$.
For $\|E_i\|\gg 1, \ i=1, 2$, it is not hard to see that the norm of $E_2E_1$ approximately  equals to $\left\|E_2\right\|\cdot\left\|E_1\right\|$ unless the angle tends to $0$.
We say that it is nonresonant if
$$
 \max\left\{\|E_1\|^{-1},\|E_2\|^{-1}\right\} \ll \left|s(E_2)-u(E_1)\right|.
$$
The following basic lemma gives the norm estimates and angle estimates for $E_3(x)=E_2(x)\cdot E_1(x)$ in the nonresonant case.
\begin{lemma}\label{lemma.basic}
Let $E_3(x)=E_2(x)\cdot E_1(x)$. Denote
$$e_j(x)=\|E_j(x)\|,\quad s_j(x)=s(E_j(x)),\quad u_j(x)=u(E_j(x)), \quad j=1,2,3,$$
$$\theta(x)=s_2(x)-u_1(x),\quad e_0=\min\{e_1,e_2\}.$$
For $x\in I$, suppose we have  $e_0\gg 1$ and  $|\theta|\geq e_0^{-\eta}$ with $0<\eta<10^{-2}$.
Then we have
$$e_3= e_1e_2|\sin\theta|+O(e_0^{-\frac12}),$$
$$|u_3(x)-u_2(x)|<e_2^{-\frac74},\quad |s_3(x)-s_1(x)|<e_1^{-\frac74},\quad x\in I.$$
\end{lemma}

Moreover, we could also obtain a $C^k$ version of Lemma \ref{lemma.basic}.
For $f\in C^k(\mathbb S^1,\mathbb {RP}^1)$, define
$$\|f\|_{C^k}:=\max_{j\in \{0,1,2,\cdots k\}}\max_{x\in \mathbb S^1}\left|\partial^j f(x)\right|.$$
\begin{lemma}\label{lemma.hderivative}
Let $e_j$, $s_j$, $u_j$ $(j=1,2,3)$, $\theta$, $e_0$ and $\eta$ be defined as in Lemma \ref{lemma.basic}.
Assume that
$$|\partial_x^ke_j|\leq C^k\cdot k!\cdot  e_j^{1+k\eta},\quad |\partial_x^k\theta|\leq e_0^{k\eta},\quad  j=1,2,\quad k=1,2,3,\cdots,l.$$
For $x\in I$, suppose $$l\eta<\frac1{100}, \quad e_0^\eta\geq C^{l^3},\quad |\theta|\geq e_0^{-\eta}.$$
Then we have for $k=1,2,3,\cdots,l$,
$$ |\partial_x^ke_3|\leq C^k\cdot k!\cdot e_3^{1+k\eta},$$
$$\|u_3-u_2\|_{C^k}\leq C^k\cdot k!\cdot e_2^{-2+(k+1)\eta},\quad \|s_3-s_1\|_{C^k}\leq C^k\cdot k!\cdot e_1^{-2+(k+1)\eta}.$$
\end{lemma}

The proofs of Lemma \ref{lemma.basic} and \ref{lemma.hderivative} are essentially given in \cite{WY13} (see the proof of Lemma 3.8), where the bounded-type condition seems involved (in fact, it does not). The present form of Lemma \ref{lemma.basic} and \ref{lemma.hderivative} could be considered as a concise version which is independent of the frequency. To keep the paper self-contained, we give simplified proofs of lemmas in the appendix.

\

On the contrary to nonresonant case, we say that it is completely resonant if the angle between two matrices vanishes, i.e.
$$s(E_2)-u(E_1)=0.$$
The next lemma, saying that the cancelation of the norms happens in the completely resonant case, is the key to prove discontinuity of LE.
\begin{lemma}\label{lemma.resonantnorm}
Suppose that $A$ and $B$ are two hyperbolic matrices such that $\|A\|=\lambda_1^m$ and $\|B\|=\lambda_2^n$ with $m,n>0$ and $\lambda_1,\lambda_2\gg 1$. If $u(A)=s(B)$, then
$$\|BA\|\leq 2\max\{\lambda_1^m\cdot\lambda_2^{-n},\lambda_1^n\cdot\lambda_2^{-m}\}.$$
\end{lemma}
\begin{proof}
From the singular value decomposition, it holds that
\beqs
BA=R_{u(B)}\left(
             \begin{array}{cc}
               \lambda_2^n & 0 \\
               0 & \lambda_2^{-n} \\
             \end{array}
           \right)R_{\frac\pi2-s(B)+u(A)}\left(
             \begin{array}{cc}
               \lambda_1^m & 0 \\
               0 & \lambda_1^{-m} \\
             \end{array}
           \right)R_{\frac\pi2-s(A)}.
\eeqs
Note that we have $s(B)=u(A)$. Then
\begin{align*}
BA&=R_{u(B)}\left(
             \begin{array}{cc}
               \lambda_2^n & 0 \\
               0 & \lambda_2^{-n} \\
             \end{array}
           \right)R_{\frac\pi2}\left(
             \begin{array}{cc}
               \lambda_1^m & 0 \\
               0 & \lambda_1^{-m} \\
             \end{array}
           \right)R_{\frac\pi2-s(A)}\\& =R_{u(B)}\left(
             \begin{array}{cc}
               0 & -\lambda_2^n\lambda_1^{-m} \\
               \lambda_1^m\lambda_2^{-n} & 0 \\
             \end{array}
           \right)R_{\frac\pi2-s(A)}.
\end{align*}
Hence we have
$$\|BA\|\leq 2\max\{\lambda_1^m\cdot\lambda_2^{-n},\lambda_1^n\cdot\lambda_2^{-m}\}.$$
\end{proof}

At the end of this section, we define $\mu$-hyperbolic sequence of SL(2,$\mathbb R$)-matrices, which is useful in Young's inductive argument.
Given a sequence of SL(2,$\mathbb R$)-matrices $$\{\cdots,A_{-2}, A_{-1},A_0,A_1,A_2\cdots\},$$ we denote $$A^n=A_{n-1}A_{n-2}\cdots A_1A_0,\quad A^{-n}=A_{-n}^{-1}A_{-n+1}^{-1}\cdots A_{-2}^{-1} A_{-1}^{-1}.$$
\begin{definition}
For any $1<\mu\leq \lambda$, we say the block of matrices $\{A_0,A_1,\cdots, A_{n-1}\}$ is $\mu$-hyperbolic if
\beq\label{hyperbolic}
\|A_i\|\leq \lambda, \quad \|A^i\|\geq \mu^{i(1-\varepsilon)},\quad \forall\  i \in \mathbb \{0,1,2,\cdots, n-1\},\eeq
and (\ref{hyperbolic}) holds if $A_0,A_1,\cdots, A_{n-1}$ are replaced by $A_{-n}^{-1},\cdots,A_{-2}^{-1},A_{-1}^{-1}$.
\end{definition}

\section{Sharp estimates of the returning times}
In this section, we will specially design a sequence of decreasing intervals $I_n$ for each $\alpha$, such that the first returning times is clear.

\subsection{Sharp distribution of the trajectory of the irrational  shift}
For any irrational $\alpha $, we consider its continued fractional expansion as follows. Define the map $f:(0,1)\to (0,1)$ by
\[f(x)=\frac{1}{x}-\left[\frac{1}{x}\right],\]
where $[x]$ is the largest integer such that $[x]\le x$. Let $\alpha_n(n\ge 0)$ be defined by
\[\alpha_0=\alpha-[\alpha],\ \alpha_n=f^n(\alpha_0)\ \mbox{for}\ n>0.\]
Set
\[a_0=[\alpha],\ a_n=\alpha^{-1}_{n-1}-\alpha_n,\ \forall n\ge 1.\]
Then, it yields that
\[\alpha=a_0+\frac{1}{a_1+\frac{1}{\ddots+\frac{1}{a_n+\alpha_n}}}\]
or $\alpha=[a_0,a_1,\cdots,a_n+\alpha_n]$ for short.

The fractional expansion of $\alpha$ is defined by
\[\frac{p_n}{q_n}=[a_0,a_1,\cdots,a_n]=a_0+\frac{1}{a_1+\frac{1}{\ddots+\frac{1}{a_n}}}.\]
It is easy to identify the numerator and denominator of the $n-th$ convergent with the sequences $\{p_n\}_{n\ge -2}$ and $\{q_n\}_{n\ge -2}$ defined by:
\[p_{-2}=0,\ p_{-1}=1,\ \mbox{and}\ p_n=a_np_{n-1}+p_{n-2},\ \forall n\ge 0,\]
\[q_{-2}=0,\ q_{-1}=1,\ \mbox{and}\ q_n=a_nq_{n-1}+q_{n-2},\ \forall n\ge 0.\]
It is well-known that $q_n\alpha-p_n>0$ for even $n$, and $q_n\alpha-p_n<0$ for odd $n$. Moreover, $q_n\alpha-p_n$ is the best approximation of $0$ as follows:
\begin{equation}\label{best-appr}
  |q_n\alpha-p_n|<|q\alpha-p|,\ \forall \ 0<q\not=q_n<q_{n+1},\ p\in  \mathbb Z.
\end{equation}

Let
$$z_n=q_n\alpha-p_n=(-1)^n\|q_n\alpha\|,$$
where
\[ \|x\|:=\min_{n\in \mathbb{Z}}|x+n|,\ \forall x\in \mathbb{R}.\]
It holds that
$$|z_{n-1}|=a_{n+1}\cdot |z_n|+|z_{n+1}|,$$
\beq\label{qzn} q_{n+1}\cdot|z_n|+q_n\cdot|z_{n+1}|=1.\eeq
Due to the fact that
$$\frac{1}{q_{n}(q_{n}+q_{n+1})}<|\alpha-\frac{p_{n}}{q_{n}}|<\frac{1}{q_{n}q_{n+1}},\quad n\in \mathbb N,$$
we have
\begin{equation}
	\label{least}\frac{1}{q_{n}+q_{n+1}}<\|q_{n}\alpha\|=|z_n|< \frac1{q_{n+1}}.
\end{equation}

 Actually, compared to (\ref{best-appr}), we have some more precise descriptions for the distribution of the set $\{q\alpha-p\}$ as follows:
\begin{lemma}\label{lem-best-appr}
Let $q\ge 0, p\in\mathbb{Z}, n\ge 0$. If $q\alpha-p$ is strictly between $q_n\alpha-p_n$ and $q_{n+1}\alpha-p_{n+1}$, then either $q\ge q_n+q_{n+1}$ or $p=q=0$.
\end{lemma}
\begin{proof}
  Without loss of generality, we assume $n$ is even. Let $x=q\alpha-p(q\not=0)$ is strictly between $q_n\alpha-p_n$ and $q_{n+1}\alpha-p_{n+1}$, i.e.,
  \begin{equation}\label{qnqqn1}
    q_{n+1}\alpha-p_{n+1}<q\alpha-p<q_n\alpha-p_n.
  \end{equation}
  Due to (\ref{best-appr}), we have
  \[q>q_{n+1}.\]

  Now, we assume $q<q_{n+1}+q_{n}$. Then, by (\ref{best-appr}) and (\ref{qnqqn1}), it implies that $x>0$ and
  \[ \left|z_{n+1}\right|=-\left(q_{n+1}\alpha-p_{n+1}\right)<x=q\alpha-p <z_n=q_n\alpha-p_n.\]
   By noting that $0<q-q_{n+1}<q_n$, we obtain from (\ref{best-appr}) that
  \[|x-z_{n+1}|=\left|\left(q-q_{n+1}\right)\alpha-\left(p-p_{n+1}\right)\right|>z_n,\]
  which implies
  \[ |z_n-x|<|z_{n+1}|.\]
  However, since $0<q-q_n<q_{n+1}$, we yield  by (\ref{best-appr}) again that
  \[ |z_n-x|=\left|\left(q-q_n\right)\alpha-\left(p-p_n\right)\right|>|z_{n+1}|.\]
\end{proof}

\begin{prop}\label{lemma.yaccoz}
Consider the set $\mathcal{X}_{n}:=\{q\alpha-p: 0\leq q<q_{n+1},p\in\mathbb Z\}$, ordered as $\cdots<x_{-1}<0=x_0<x_1<x_2<\cdots$. Let $x_k=q\alpha-p$ be an element of this set.
\begin{itemize}
  \item If $n$ is even, then $$x_{k+1}=\left\{
            \begin{array}{ll}
              x_k+|z_n|, & \hbox{$q<q_{n+1}-q_n$;} \\
              x_k+|z_n|+|z_{n+1}|, & \hbox{$q\geq q_{n+1}-q_n$.}
            \end{array}
          \right.$$
  \item If $n$ is odd, then $$x_{k-1}=\left\{
            \begin{array}{ll}
              x_k+|z_n|, & \hbox{$q<q_{n+1}-q_n$;} \\
              x_k+|z_n|+|z_{n+1}|, & \hbox{$q\geq q_{n+1}-q_n$.}
            \end{array}
          \right.$$
\end{itemize}
\end{prop}
\begin{proof}
  We also consider only the case $n$ is even, since the odd case is similar.

  Let $q<q_{n+1}-q_n$. We declare that $x_{k+1}=x_{k}+z_n=(q+q_n)\alpha-(p+p_n)$. Firstly, since $q+q_n<q_{n+1}$, $(q+q_n)\alpha-(p+p_n)\in \mathcal{X}_n$. Secondly, if there exists some $x'_{k+1}=q'\alpha-p'<x_k +z_n$ with $0\le q'<q_{n+1}$, it will be a contradiction by (\ref{best-appr}) that $|x'_{k+1}-x_k|=|(q'-q)\alpha-(p'-p)|<z_n$.

  Let $q\ge q_{n+1}-q_n$. We declare that $$x_{k+1}=x_{k}+|z_n|+|z_{n+1}|=x_{k}+z_n-z_{n+1}=(q+q_n-q_{n+1})\alpha-(p+p_n-p_{n+1}).$$ Firstly, since $0\le q+q_n-q_{n+1}<q_{n}$, $(q+q_n-q_{n+1})\alpha-(p+p_n-p_{n+1})\in \mathcal{X}_n$. Secondly, if there exists some $x'_{k+1}=q'\alpha-p'<x_{k}+|z_n|+|z_{n+1}|$ with $0<q'<q_{n+1}$, it implies that $x'_{k+1}$ is between $x_k=\left(x_k-z_{n+1}\right)+z_{n+1}$ and $x_{k+1}=\left(x_k-z_{n+1}\right)+z_{n}$. Due to the fact that $x_k-z_{n+1}=\left(q-q_{n+1}\right)\alpha-\left(p-p_{n+1}\right)$, we have by Lemma \ref{lem-best-appr} that
  \[q'-\left(q-q_{n+1}\right)\ge q_n+q_{n+1}.\] It will also be a contradiction  with $0<q'<q_{n+1}$.
\end{proof}
In summary, these points $x_k$ can be expressed more concise as follow
\begin{corollary}
  \[x_k\equiv l\alpha \ (\mathrm{mod}\ 1  ),\ \mbox{where}\ l\equiv kq_n \ (\mathrm{mod } \ q_{n+1}),\ 0\le k\le q_{n+1}-1.\]
\end{corollary}

From Proposition \ref{lemma.yaccoz}, we could derive a useful estimate for the upper bound of the first entering time to an interval.

\begin{lemma}\label{liouville}
Let $x$ be arbitrary points in $\mathbb S^1$.
Denote $b_n=(|z_n|+|z_{n+1}|)/2$ and $I:=[-b_n,b_n)$.
Then
$$
\min\left\{j\in \mathbb N:T^jx\in I\right\}<q_{n+1}.
$$
In other words, the largest time of any trajectory entering $I$ is smaller than $q_{n+1}$.
\end{lemma}

\begin{proof}
By Proposition \ref{lemma.yaccoz} and (\ref{qzn}), the trajectory $\{T^jx:0\leq j<q_{n+1}\}$ approximates to a sequence of points even distributed on $[0,1]$ with distance between adjacent points being $|z_n|$ or $|z_n|+|z_{n+1}|$. Thus there must be a $j_0\in \{0,1,\cdots, q_{n+1}-1\}$ such that $T^{j_0}x\in I$, which proves the lemma.
\end{proof}

\subsection{Sharp estimates of the returning times for a single interval.}
In this subsection, we will create some intervals with suitable length, such that the returning time of these intervals are very concise and simple.

For even $n$, we
define
\begin{eqnarray*}
I_n^0&=&[0,|z_{n+1}|),\\
I_n^i&=&[|z_{n}|-(i-1)|z_{n+1}|,|z_{n}|-(i-2)|z_{n+1}|),\quad 1\leq i\leq a_{n+2},\\
I_n^*&=&[|z_{n+1}|,|z_{n+1}|+|z_{n+2}|).
\end{eqnarray*}
It is easy to see that
$$I_n := [0,|z_n|+|z_{n+1}|)=I_n^0\ \bigcup\  I_n^*\ \bigcup\ \left(\bigcup_{i=1}^{a_{n+2}} I_n^i\right).$$

\begin{lemma}\label{lemma.returntimeIn}
For $x\in I_n$, let $r_n^+(x)$ be the first forward returning time to $I_n$, that is, $r_n^+(x)=\min\{j\in \mathbb Z_+: T^jx\in I_n\}$. Then it holds that
$$r_n^+(x)=\left\{
            \begin{array}{ll}
              q_n, & \hbox{$x\in I_n^0$;} \\
              q_{n+1}, & \hbox{$x\in I_n\backslash I_n^0$.}
            \end{array}
          \right.
$$
\end{lemma}
\begin{proof}
If $x\in I_n^0=[0,-z_{n+1})$, then
$$x+z_n\in [z_n,-z_{n+1}+z_n)= I_n^1 \subset I_n.$$
On the other hand,  by (\ref{least}) for $0<j<q_n$, $\|j\alpha\|\geq \|q_{n-1}\alpha\|=|z_{n-1}| $. Note that for any irrational $\alpha$, $|z_{n-1}|\ge |z_n|+|z_{n+1}|$. So, due to Proposition \ref{lemma.yaccoz} and
the fact that $z_{n-1}<0$, we have $x+j\alpha\notin I_n$. This shows that for $x\in I_n^0$, $r_n^+(x)=q_n$.

If $x\in I_n^i=[|z_{n}|-(i-1)|z_{n+1}|,|z_{n}|-(i-2)|z_{n+1}|)\subset I_n\backslash I_n^0, \  1\leq i\leq a_{n+2}$, then
$$x+z_{n+1}\in [|z_{n}|-i|z_{n+1}|,|z_{n}|-(i-1)|z_{n+1}|) \subset I_n.$$
On the other hand, by (\ref{least}) for $0<j<q_{n+1}$, $\|j\alpha\|\geq \|q_n\alpha\|=z_n$. However,
$$x+z_n\in [2|z_{n}|-(i-1)|z_{n+1}|,2|z_{n}|-(i-2)|z_{n+1}|)\cap I_n=\varnothing.$$
By Proposition \ref{lemma.yaccoz}, it holds that for any $x\in  I_n^i$, $r_n^+(x)=q_{n+1}$.

If $x\in I_n^*=[|z_{n+1}|,|z_{n+1}|+|z_{n+2}|)\subset I_n\backslash I_n^0$, then
$$x+z_{n+1}\in [0,|z_{n+2}|) \subset I_n^0 \subset I_n,$$
and
$$x+z_n\in [|z_n|+|z_{n+1}|,|z_n|+|z_{n+1}|+|z_{n+2}|)\cap I_n=\varnothing.$$
Thus,  for any $x\in I_n^*$, $r_n^+(x)=q_{n+1}$.
\end{proof}
From the proof, we have more concise information of the first returning.
\begin{corollary}\label{cor.sharp-est-rt}
\begin{eqnarray*}
  T^{q_n}(I_n^0)&=&I_n^1\\
  T^{q_{n+1}}(I_n^i)&=&I_n^{i+1},\ i=1,2,\cdots, a_{n+2}-1\\
  T^{q_{n+1}}(I_n^{a_{n+2}})&\subset&I_n^{*}\cup I_n^0\\
  T^{q_{n+1}}(I_n^{*})&\subset& I_n^0.
\end{eqnarray*}
\end{corollary}~\\

Our improvement heavily relies on the special choice of $I_n$ for each $\alpha$, which makes the first returning time is concise. Since $\alpha$ is irrational, it is impossible to have constant first  returning time for any interval. So Lemma
\ref{lemma.returntimeIn} is the best outcome one can expect. The following lemma tells us that $I_n^0$ is similar to $I_n$:
\begin{lemma}\label{lemma.returntimeIn0}
For $x\in I_n^0$, the returning time to $I_n^0$ forward is either $q_{n+2}$ or $q_{n+2}+q_{n+1}$.
\end{lemma}
\begin{proof}
Due to Proposition \ref{lemma.yaccoz} and the fact that $|I_n^0|=z_{n+1}$, we obtain that for any $x\in I_n^0$, the returning times could not be smaller than $q_{n+2}$.

Define
\begin{eqnarray*}
  I_n^{0,i}&=&\left[(i-1)|z_{n+2}|, i|z_{n+2}|\right),\ i=1,\cdots, a_{n+3}-1,\\
  I_n^{0,*}&=&\left[(a_{n+3}-1)|z_{n+2}|, (a_{n+3}-1)|z_{n+2}|+|z_{n+3}|\right)\\
  I_n^{0,0}&=&\left[(a_{n+3}-1)\left|z_{n+2}\right|+|z_{n+3}|, \left|z_{n+1}\right| \right).
\end{eqnarray*}
It is easy to check that
\begin{eqnarray*}
 T^{q_{n+2}}\left(I_n^{0,i}\right)&=&\left[i|z_{n+2}|, (i+1)|z_{n+2}|\right)=I_n^{0,i+1}\subset I_n^0,\ i=1,\cdots, a_{n+3}-2,\\
 T^{q_{n+2}}\left(I_n^{0, a_{n+3}-1}\right)&=&\left[(a_{n+3}-1)|z_{n+2}|, a_{n+3}|z_{n+2}|\right)\\
 &=&\left[z_{n+1}-|z_{n+3}|-|z_{n+2}|, z_{n+1}-|z_{n+3}|\right)\subset I_n^{0,*}\cup I_n^{0,0}\subset  I_n^0,\\
 T^{q_{n+2}}\left(I_n^{0, *}\right)&=&\left[a_{n+3}\left|z_{n+2}\right|, a_{n+3}\left|z_{n+2}\right|+|z_{n+3}| \right)\\
 &=&\left[\left|z_{n+1}\right|-\left|z_{n+3}\right|, \left|z_{n+1}\right| \right)\subset I_n^{0,0}\subset  I_n^0,
\end{eqnarray*}
which implies that the  returning time is $q_{n+2}$ for any $x\in I_n^0 \backslash I_n^{0,0}$.

At last, for the interval $I_n^{0,0}$, it has that
\begin{eqnarray*}
  T^{q_{n+2}}\left(I_n^{0,0}\right)&=&\left[a_{n+3}\left|z_{n+2}\right|+|z_{n+3}|, \left|z_{n+1}\right|+\left|z_{n+2}\right| \right)\\
 &=&[\left|z_{n+1}\right|,\ \left|z_{n+1}\right|+\left|z_{n+2}\right|\cap I_n^0=\varnothing,\\
 T^{q_{n+2}+q_{n+1}}\left(I_n^{0,0}\right)&=&\left[0, \left|z_{n+2}\right| \right)=I_n^{0,1}\subset  I_n^0.
\end{eqnarray*}
Thus, due to Lemma \ref{lem-best-appr}, we obtain that the returning time is $q_{n+2}+q_{n+1}$ for any $x\in  I_n^{0,0}$.
\end{proof}

It is easy to check that  Lemma \ref{lemma.returntimeIn}, Corollary \ref{cor.sharp-est-rt} and Lemma \ref{lemma.returntimeIn0} will also be valid for odd $n$, if we redefine the intervals as follow:
\begin{eqnarray*}
  I_n&=&(0,|z_n|+|z_{n+1}|],\\
I_n^0&=&(|z_n|,|z_n|+|z_{n+1}|],\\
I_n^i&=&((i-1)|z_{n+1}|,i|z_{n+1}|],\quad 1\leq i\leq a_{n+2},\\
I_n^*&=&(|z_{n}|-|z_{n+2}|,|z_{n}|].
\end{eqnarray*}

If we consider the first backward returning time
$$r_n^-(x)=\min\{j\in \mathbb Z_+: T^{-j}x\in I_n\},$$
then we have similar results as  Lemma \ref{lemma.returntimeIn}, Corollary \ref{cor.sharp-est-rt} and Lemma \ref{lemma.returntimeIn0}.

\section{Proof of Theorem \ref{thm.main1}: Discontinuity of LE in $C^l$ topology}
In this section, we prove Theorem \ref{thm.main1} with two convergent sequences of cocycles $\{A_n\}$, which have increasing norms in inductive process and converges to $D_l$ in $C^l$ topology, and $\{\widetilde A_n\}$, which have cancelation on norms and also converges to $D_l$ in $C^l$ topology.
Under finite-Liouvillean condition on frequency, we prove positivity of $L(D_l)$ by Young's inductive argument and then derive discontinuity of $L(D_l)$ through cancelation on norms and sharp estimates of returning times.

Firstly, we give the following settings.
\begin{itemize}
  \item The frequency: Let $\alpha$ be a finite-Liouvillean number, i.e.
      \beqs\beta(\alpha) :=\limsup_{n\rightarrow\infty} \frac{\log q_{n+1}}{q_n}<\infty,\eeqs
      where $p_n/q_n$ be the fractional expansion of $\alpha$.
  \item The critical interval: $I_{n}=[-b_n,b_n)$, where $b_n=\frac{|z_n|+|z_{n+1}|}{2}$ and $z_n=q_n\alpha-p_n$.
  \item The first returning time: For $x\in I_n$, we denote the smallest positive integer $i$ with $T^ix\in I_n$ (respectively $T^{-i}x\in I_n$) by $r_n^{+}(x)$ (respectively $r_n^-(x)$).
  \item The sample function: Fix $\delta_0>0$ small and $l\in \mathbb Z_+$. The $2\pi$-periodic smooth function $\phi_0\in C^{l}(\mathbb S^1)$ is defined as
        $$\phi_0(x)=\left\{
                 \begin{array}{ll}
                   |x|^{l+1}, & \hbox{$|x|\leq\delta_0$;} \\
                   \text{such\ that\ } \delta_0^{l+1}<|\phi_0(x)|<\frac\pi2,  & \hbox{$|x|>\delta_0$.}
                 \end{array}
               \right.
$$
\end{itemize}

For any $C\geq 1$, we denote by $\frac{I_{n}}{C}$ the sets $[-\frac{b_n}{C},\frac{b_n}{C})$.
Let $\Lambda=\left(
               \begin{array}{cc}
                 \lambda & 0 \\
                 0 & \lambda^{-1} \\
               \end{array}
             \right)
$.
Fix $\varepsilon>0$ arbitrary small.
Let $N$ be sufficiently large such that
\beq\label{def.N} b_N\ll \delta_0,\quad \sum_{n>N} (l+1)q_{n+1}^{-\frac12} \leq\frac\varepsilon8.\eeq
Assume that $\lambda$ is sufficiently large such that
\beq\label{def.lambda}\lambda>q_{N+1}^{100\varepsilon^{-1}(l+1)}.\eeq
Recall that $q_{N+1}\leq b_N^{-1}\leq 4q_{N+1}$. Hence $\lambda>b_N^{-50\varepsilon^{-1}(l+1)}$.
Inductively, we could define $\lambda_n$ ($n\geq N$) as
\beq\label{def.lambdan}
\log\lambda_n:=\left\{
             \begin{array}{ll}
               \log\lambda +(l+1)\log b_N, & \hbox{$n=N$;} \\
               \Big(1-2(l+1)q_{n-1}^{-\frac12}\Big)\log\lambda_{n-1}, & \hbox{$n\geq N+1$.}
             \end{array}
           \right.
\eeq
Clearly, the sequence $\lambda_n$ decreases to some $\lambda_\infty$ with $\lambda_\infty>\lambda^{1-\frac\varepsilon4}$.

With above settings, we construct two sequences of cocycles in the following propositions, whose proofs will be given in the next section.

\begin{prop}\label{prop.lowerbound}
There exists $\lambda_0=\lambda_0(l,\alpha, \phi_0)$ such that for $\lambda>\lambda_0$ the following holds.
There exist functions $\phi_n(x)$ on $\mathbb S^1$ ($n=N,N+1,\cdots$) such that
\begin{description}
  \item[$(1)_n$] The function $\phi_n(x)$ is exactly the same as $\phi_{n-1}(x)$ for $x\notin I_n$, and moreover it holds that $$\big\|\phi_n(x)-\phi_{n-1}(x)\big\|_{C^l}\leq \lambda_n^{-\frac13q_{n-1}},\quad n>N.$$
  \item[$(2)_n$] The sequence $\{A_n(x), A_n(Tx),\cdots,A_n(T^{r_n^+-1}x)\}$ is $\lambda_n$-hyperbolic for $x\in I_n$ where $A_n(x):=\Lambda R_{\frac\pi2-\phi_n(x)}$. Moreover, for $x\in \frac{I_n}{10}$, $1\leq k \leq l$ and $0<\eta<\frac{1}{100 \cdot  l}$, $$\left|\partial^k \|A_n^{r_n^+}(x)\|\right|\leq \|A_n^{r_n^+}(x)\|^{1+k\eta}.$$
  \item[$(3)_n$] Let $s_n(x)={s(A_n^{r_n^+}(x))}$ and $u_n(x)={s(A_n^{-r_n^-}(x))}$. Then it holds that
\begin{description}
  \item[$(3.1)_n$] $s_n(x)-u_n(x)=\phi_0(x)$, for $x\in \frac{I_n}{10}$;
  \item[$(3.2)_n$] $|s_n(x)-u_n(x)|\geq \frac12 |\phi_0(x)|\geq \left(\frac{b_n}{20}\right)^{l+1}$, for $x\in I_n\backslash \frac{I_n}{10}$;
  \item[$(3.3)_n$]  $|\partial^k( s_n(x)-u_{n}(x))|\leq  \lambda^{q_n\cdot\eta}$, for $x\in I_n$ and $1\leq k \leq l$.
\end{description}
\end{description}
\end{prop}

\begin{prop}\label{prop.upperbound}
There exists $\lambda_0=\lambda_0(l,\alpha, \phi_0)$ such that for $\lambda>\lambda_0$ the following holds.
There exist functions $\widetilde\phi_n(x)$ on $\mathbb S^1$ ($n=N,N+1,\cdots$) such that
\begin{description}
  \item[$(1')_n$] The function $\widetilde\phi_{n}$ converges to $\phi_n$ in $C^l$ topology $$\big\|\phi_n(x)-\widetilde\phi_{n}(x)\big\|_{C^l}\leq q_{n+1}^{-2},\quad n>N.$$
  \item[$(2')_n$] The sequence $\{\widetilde A_n(x), \widetilde A_n(Tx),\cdots,\widetilde A_n(T^{r_n^+-1}x)\}$ is $\lambda_n$-hyperbolic for $x\in I_n$, where $\widetilde A_n(x):=\Lambda R_{\frac\pi2-\widetilde \phi_n(x)}$.
  \item[$(3')_n$] Let $\widetilde s_n(x)={s(\widetilde A_n^{r_n^+}(x))}$ and $\widetilde u_n(x)={s(\widetilde A_n^{-r_n^-}(x))}$. Then it holds that
      $$\widetilde s_n(x)=\widetilde u_n(x),\quad x\in \frac{I_n}{10}.$$
\end{description}
\end{prop}
By $(1)_n$ in Proposition \ref{prop.lowerbound}, $(1')_n$ in Proposition \ref{prop.upperbound} and the completeness of $C^l(\mathbb S^1,\mathrm{SL}(2,\mathbb R))$, there exists a $D_l\in C^l(\mathbb S^1,\mathrm{SL}(2,\mathbb R))$ such that as $n$ goes to infinity,
$$\|A_n-D_l\|_{C^l}\rightarrow 0,\quad \|\widetilde A_n-D_l\|_{C^l}\rightarrow 0.$$
Thus to prove Theorem \ref{thm.main1}, it suffices to prove
\begin{itemize}
  \item The lower bound: $L(D_l)\geq (1-\varepsilon)\log \lambda$ with $\varepsilon$ very small;
  \item The upper bound: $L(\widetilde A_n)\leq (1-\delta)\log \lambda$ for sufficiently large $n$, with $\delta\geq \frac1{300}\gg \varepsilon$.
\end{itemize}

\subsection{The lower bound of LE}
Fix $\varepsilon>0$ arbitrarily small.
Let
\begin{equation}\label{exceptionset}
B_{n}=\bigcup_{l=1}^{ [q_{n+1}^{3/2}]}\left(B(0,q_{n+1}^{-2})-l\alpha\right),\quad n\geq N.
\end{equation}
By (\ref{def.N}), it is straightforward to compute that
$$ \text{meas}\bigg(\bigcup_{n\geq N}B_{n}\bigg)\leq \sum_{n>N} q_{n+1}^{-\frac12} \leq\frac\varepsilon4,$$
for $N$ large enough.

Now we are going to estimate the norm of $A_m^{i}(x)$ with $m\geq n$ and $i=[q_{n+1}^{3/2}]$ for $x\notin \bigcup_{n\geq N}B_{n}$.
Consider the trajectory $\{T^jx:0\leq j\leq i\}$, let $i_0=i_0(x)\geq 0$ be the first returning time to $I_{n}$, and $i_1=i_1(x)\geq 0$ be the last returning time to $I_{n}$.
By Lemma \ref{liouville}, we have that $$i_0,i-i_1\leq q_{n+1}\ll i.$$
By (\ref{exceptionset}), we know that the trajectory $\{T^jx:0\leq j\leq i\}$ has not entered $B(0,q_{n+1}^{-2})$, but has entered $I_{n}$.
Let $i_0=j_0<j_1<\cdots<j_p=i_1$ be the returning times to $I_{n}$.
Now we consider the decomposition
$$A_m^i(x)=A_m^{i-i_1}(T^{i_1}x)\cdot A_m^{i_1-i_0}(T^{i_0}x) \cdot A_m^{i_0}(x),$$
\beq\label{decomposition} A_m^{i_1-i_0}(T^{i_0}x)=A_m^{j_{p}-j_{p-1}}(T^{j_{p-1}}x)\cdots A_m^{j_{2}-j_{1}}(T^{j_{1}}x)A_m^{j_{1}-j_{0}}(T^{j_{0}}x).\eeq
Since the trajectory has not entered $B(0,q_{n+1}^{-2})$, we have
$$
 \text{dist}\left(T^{j_l}x,0\right)>q_{n+1}^{-2}.
$$
Hence it holds that
$$
 |s_n(T^{j_l}x)-u_n(T^{j_l}x)|>\frac12|\phi_0(T^{j_l}x)|\gtrsim q_{n+1}^{-2(l+1)}.
$$

Due to the returning properties of $I_{n}$ in Lemma \ref{lemma.returntimeIn}, it holds that $j_{t}-j_{t-1}\geq q_{n}$.
By $(2)_n$ in Proposition \ref{prop.lowerbound},
\beqs
\|A_m^{j_{t}-j_{t-1}}(T^{j_{t-1}}x)\|\geq \lambda_m^{j_t-j_{t-1}}\geq \lambda_m^{q_n},\quad 1\leq t\leq p.
\eeqs
Note that $\alpha$ is a finite-Liouvillean number, i.e. $q_{n+1}\leq C e^{\beta q_n}$ for sufficiently large $n$. Then it holds for $\lambda$ sufficiently large,
\beq\label{normbigangle}\|A_m^{j_{t}-j_{t-1}}(T^{j_{t-1}}x)\|\geq \lambda_m^{q_n} \gg q_{n+1}^{2(l+1)}.\eeq
Then repeatedly applying Lemma \ref{lemma.basic} to (\ref{decomposition}) with (\ref{normbigangle}), we obtain
\begin{align*}\|A_m^{i_1-i_0}(T^{i_0}x)\|
&\geq \prod_{t=1}^p\|A_m^{j_{t}-j_{t-1}}(T^{j_{t-1}}x)\|\cdot \prod_{t=1}^{p-1}\left(\frac12|\phi_0(T^{j_t}x)|-\lambda_m^{-q_{n}}\right)\\
&\geq \prod_{t=1}^p\lambda_m^{j_{t}-j_{t-1}}\cdot C^{p-1}\cdot \prod_{t=1}^{p-1}\big(\text{dist}\left(T^{j_t}x,0\right)\big)^{l+1}\\
&\geq \lambda_m^{j_{p}-j_{0}}\cdot C^{p-1}\cdot q_{n+1}^{-2(l+1) (p-1)}.
\end{align*}
Due to the returning times of $I_{n}$ given in Lemma \ref{lemma.returntimeIn} (In $[0,q_{n+2}]$, most of them are $q_{n+1}$ and at most one is $q_{n}$), we have $$1\leq p\leq 2[q_{n+1}^{1/2}].$$
This implies that for $x\notin B_{n}$,
\beas
\|A_n^{i}(x)\|&\geq & \|A_n^{i_1-i_0}(T^{i_0}x)\|\cdot \|A_n^{i-i_1}(T^{i_1}x)\|^{-1}\cdot\|A_n^{i_0}(T^{i_0}x)\|^{-1}\\
&\geq & \lambda_m^{j_{p}-j_{0}}\cdot C^{2q_{n+1}^{1/2}}\cdot q_{n+1}^{-4(l+1) q_{n+1}^{1/2}}\cdot \lambda^{-2q_{n+1}}\\
&\geq & \lambda_m^{j_{p}-j_{0}}\cdot \lambda^{i\cdot\left(-4(l+1)\frac{\log q_{n+1}}{\log\lambda \cdot q_{n+1}}+\frac{2\log C}{\log\lambda \cdot q_{n+1}}\right)}\cdot \lambda^{-2q_{n+1}}\\
&\geq & \lambda_m^{i}\cdot \lambda^{i\cdot\left(-4(l+1)\frac{\log q_{n+1}}{\log\lambda \cdot q_{n+1}}+\frac{2\log C}{\log\lambda \cdot q_{n+1}}-4q_{n+1}^{-1/2}\right)}\\
&\geq & \lambda_\infty^{i},
\eeas
where we have used (\ref{def.N}) and (\ref{def.lambdan}). Note $i=[q_{n+1}^{3/2}]$.
Hence we have
\beq\label{lowerbound}
\frac1{[q_{n+1}^{3/2}]}\log \big\|A_m^{[q_{n+1}^{3/2}]}(x)\big\|\geq (1-\frac\varepsilon4)
\log\lambda, \quad x\notin B_{n},\quad m\geq n.\eeq

Now we estimate the lower bound of $L(D_l)$. By subadditivity, the finite LE of $D_l$ converges. Hence there exists a large $N_0$ such that for $N_1\geq N_0$,
$$\left|\frac1{N_1}\int_{\mathbb S^1}\log\big\|D_l^{N_1}(x)\big\| dx-L(D_l)\right|\leq \frac\varepsilon8.$$
Fix $N_1=[q_{n+1}^{3/2}]$ for some large $n$ such that $N_1\geq N_0$.
Since $A_m$ converges to $D_l$ in $C^l$ topology as $m\rightarrow \infty$, there exists a large $N_2>n$ such that for any $m>N_2$,
$$\left|\frac1{N_1}\int_{\mathbb S^1}\log\big\|D_l^{N_1}(x)\big\| dx-\frac1{N_1}\int_{\mathbb S^1}\log\left\|A_m^{N_1}(x)\right\| dx\right|\leq \frac\varepsilon8.$$
On the other hand, by (\ref{lowerbound}) we have
$$
\frac1{N_1}\int_{\mathbb S^1}\log\left\|A_m^{N_1}(x)\right\| dx \geq (1-\frac\varepsilon4)^2\log\lambda-\frac\varepsilon4 \log \lambda\geq (1-\frac34\varepsilon)\log\lambda.
$$
Hence we could conclude
$$L(D_l)\geq (1-\varepsilon)\log\lambda.$$

\subsection{The upper bound of LE}
Suppose that $\alpha$ is not bounded type (otherwise, everything was proved in \cite{WY13}). Then there exist infinitely many $n$'s satisfying $n>N$ and
\beq\label{notfinitetype}q_{n+2}\geq 200q_{n+1}.\eeq
Let $\cdots<m_{j-1}<m_j<m_{j+1}<\cdots$ be the returning times of $x\in I_n^0:=[-b_n,-b_n+|z_{n+1}|)$. By Lemma \ref{lemma.returntimeIn0}, it holds that
$$m_{j+1}-m_j= q_{n+2}\quad \text{or}\quad q_{n+2}+q_{n+1}.$$
Let $m_j=n_{j_0}<n_{j_1}<n_{j_2}<\cdots<n_{j_{p-1}}<n_{j_p}=m_{j+1}$ be the returning times of $I_n$. Here $p=a_{n+2}$ or $a_{n+2}+1$.
By Lemma \ref{lemma.returntimeIn}, it holds that
$$n_{j_1}-n_{j_0}=q_{n},\quad n_{j_i}-n_{j_{i-1}}=q_{n+1},\ 2\leq i\leq p.$$
Also note that for $2\leq i\leq p$, $|T^{n_{j_i}}x-T^{n_{j_{i-1}}}x|=|z_{n+1}|$.
Thus there exist integers $s$ and $t$ such that $2\leq s\leq p$, $2\leq t\leq p$, $t-s\geq \frac p {15}$ and
$$T^{n_i}x\in \frac{I_{n}}{10},\quad \forall  i\in [s,t]\cap \mathbb Z.$$
By (\ref{notfinitetype}), we have $a_{n+2}\geq 150$, and hence $t-s\geq 10$.
Now for any $i\in [s,t-1]\cap \mathbb Z$, we are going to estimate $\|\widetilde A_n^{n_{i+2}-n_i}(T^{n_i}x)\|$.
By Proposition \ref{prop.upperbound}, it holds that
$$\lambda_n^{q_{n+1}}\leq \|\widetilde A_n^{n_{i+1}-n_i}(T^{n_i}x)\|\leq \lambda^{q_{n+1}},$$
$$ \lambda_n^{q_{n+1}} \leq \| \widetilde A_n^{n_{i+2}-n_{i+1}}(T^{n_{i+1}}x)\|\leq \lambda^{q_{n+1}},$$
$$s(\widetilde A_n^{n_{i+1}-n_i}(T^{n_i}x))-u(\widetilde A_n^{n_{i+2}-n_{i+1}}(T^{n_{i+1}}x))=0.$$
Then we could conclude from Lemma \ref{lemma.resonantnorm} that
$$\|\widetilde A_n^{n_{i+2}-n_i}(T^{n_i}x)\|\leq \lambda^{\varepsilon q_{n+1}}.$$
This implies that
$$\|\widetilde A_n^{n_{j_t}-n_{j_s}}(T^{n_{j_s}}x)\|\leq \lambda^{([\frac{t-s}2]\varepsilon+1) q_{n+1}}\leq \lambda^{2 q_{n+1}}.$$
We then consider $$\widetilde A_n^{m_{j+1}-m_j}(T^{m_j}x)=\widetilde A_n^{m_{j+1}-n_{j_t}}(T^{n_{j_t}}x) \cdot \widetilde A_n^{n_{j_t}-n_{j_s}}(T^{n_{j_s}}x)\cdot \widetilde A_n^{m_{j+1}-m_j}(T^{m_j}x).$$
Its norm is smaller than
$$\|\widetilde A_n^{m_{j+1}-m_j}(T^{m_j}x)\|\leq \lambda^{q_n+(p-t+s)q_{n+1}+2q_{n+1}}\leq \lambda^{\frac{299}{300}(m_{j+1}-m_j)} ,$$
where we used $$(t-s)q_{n+1}\geq \frac{pq_{n+1}}{15}\geq \frac{q_{n+2}}{30}\geq \frac{m_{j+1}-m_{j}}{60}$$ in the last inequality.
Note that $j$ is arbitrary.
This implies
$$L(\widetilde A_n)\leq (1-\delta)\log \lambda,\quad \delta=\frac1{300}.$$

\section{Proof of Proposition \ref{prop.lowerbound} and \ref{prop.upperbound}.}
In this section, we construct two convergent sequences of cocycles $\{A_n\}$ and $\{\widetilde A_n\}$.
Firstly, we inductively construct $\{A_n\}$ with increasing iterated norms through gradually adjustments on the angles in each step.
Secondly, we perturb the cocycles $\{A_n\}$ to another convergent sequence of cocycles $\{\widetilde A_n\}$ such that the angle between adjacent blocks vanishes, which leads to cancelation on norms.
We stress that under weak Liouvillean condition on frequency, we require sharp estimates on the angles and norms, which are given by the precise distribution of trajectory.

\subsection{Proof of Proposition \ref{prop.lowerbound}.}
We are going to prove the proposition by induction.
\subsubsection{The $N$th step.}
Let $A(x)=\Lambda\cdot R_{\frac{\pi}{2}-\phi_0(x)}$.
Consider the decomposition
$$A^{r_N^+}(x)=A(T^{r_N^+-1}x)\cdots A(Tx)A(x).$$
Note that $T^jx\notin I_N$ for $1\leq j\leq r_N^+(x)$ and hence $$\text{dist}(T^jx,0)\geq b_N,\quad |\phi_0(T^jx)| \geq b_N^{l+1}.$$
Moreover, by (\ref{def.lambda}), we have
$$\|A(T^jx)\|= \lambda \geq b_N^{-100(l+1)}\gg b_{N}^{-(l+1)}.$$
Then we successively apply Lemma \ref{lemma.basic} to $A^{j}(x)=A(T^{j-1}x)A^{j-1}(x)$ for $j=2,3,\cdots, r_N^+$.
Then we have
$$\|A^{j}(x)\|\geq \lambda^{j}\cdot b_N^{(l+1)j}= \lambda_N^j,\quad 2\leq j\leq r_N^+.$$
Hence it holds that
\beq\label{Ahyperbolic}
\{A(x),A(Tx),\cdots A(T^{r_N^+-1}x)\}\ \text{is}\ \lambda_N-\text{hyperbolic}, \quad \forall x\in I_N.
\eeq
Let $\overline s_N (x):= {s(A^{r_N^+}(x))}$ and $\overline u_N (x):= {u(A^{-r_N^-}(x))}$.
Moreover, by repeatedly applying Lemma \ref{lemma.basic} and Lemma \ref{lemma.hderivative}, we have
\beq\label{sNphi0uN}
\|\overline s_N-\phi_0\|_{C^k}\leq \lambda^{-2+k\eta},\quad \|\overline u_N\|_{C^k}\leq \lambda^{-2+k\eta},
\eeq
\beq\label{panormA}
\left|\partial^k \|A^{r_N^+}(x)\|\right|\leq \|A^{r_N^+}(x)\|^{1+k\eta},\quad 1\leq k\leq l.
\eeq

Define $e_N(x)$ to be the following $2\pi$-periodic function:
$$
e_N(x)=\left\{
         \begin{array}{ll}
           \phi_0(x)-(\overline s_N-\overline u_N)(x), & \hbox{$x\in \frac{I_N}{10}$;} \\
           h_N^\pm(x), & \hbox{$x\in I_N\backslash \frac{I_N}{10}$;} \\
           0, & \hbox{$x\in \mathbb S^1\backslash I_N$,}
         \end{array}
       \right.
$$
where $h_N^\pm(x)$ is a $C^l$-function satisfies
$$\frac{d^jh_N^\pm}{dx^j}\left(\pm\frac{b_N}{10}\right)= \frac{d^j\phi_0}{dx^j}\left(\pm\frac{b_N}{10}\right) +\frac{d^j(\overline s_N-\overline u_N)}{dx^j}\left(\pm\frac{b_N}{10}\right),$$
$$\frac{d^jh_N^\pm}{dx^j}\left(\pm b_N\right)= 0,\quad i=1,2,\quad j=0,1,2,\cdots,l.$$
From (\ref{sNphi0uN}), we have for $x\in \frac{I_N}{10}$ and $1\leq k\leq l$, $\|e_N(x)\|_{C^k}\leq \lambda^{-2+k\eta}$.
By Cramer's rules, we have
$$\|h_N^\pm(x)\|_{C^k}\leq \lambda^{-2+k\eta},\quad 1\leq k\leq l.$$
Hence it holds that for $x\in \mathbb S^1$ and $1\leq k\leq l$,
\beq\label{eNnorm}\|e_N(x)\|_{C^k}\leq \lambda^{-2+k\eta}.\eeq

Now define
$$\phi_N(x)=\phi_0(x)+e_N(x),\quad A_N(x)=\Lambda\cdot R_{\frac{\pi}{2}-\phi_N(x)},$$
$$s_N(x)={s(A_N^{r_N^+}(x))},\quad u_N(x)={s(A_N^{-r_N^-}(x))}.$$
Then by (\ref{eNnorm}) we have
$$\|\phi_N-\phi_0\|_{C^k}=\|e_N\|_{C^k}\leq \lambda^{-2+k\eta},\quad x\in \mathbb S^1.$$
Note that for $1\leq i\leq r_N^+-1$ and $x\in I_N$, $T^ix\in \mathbb S^1\backslash I_N$. This implies $A_N(x)=A(x)$ for $x\notin I_N$. Then we have for $x\in I_N$,
\beq\label{ANrN+}A_N^{r_N^+}(x)=A^{r_N^+}(x)\cdot \left(A^{-1}(x)A_N(x)\right)= A^{r_N^+}(x)\cdot R_{-e_N(x)}.\eeq
Similarly, it holds that
\beq\label{ANrN-}A_N^{-r_N^-}(x)=R_{-e_N(T^{-r_N^-}x)}\cdot  A^{-r_N^-}(x),\quad x\in I_N.\eeq
Note that rotation does not change the norm of a vector. Then combining (\ref{Ahyperbolic}), (\ref{panormA}), (\ref{ANrN+}) and (\ref{ANrN-}), we have
\beqs
\{A_N(x),A_N(Tx),\cdots A_N(T^{r_N^+-1}x)\}\ \text{is} \ \lambda_N-\text{hyperbolic},\quad  \forall x\in I_N;
\eeqs
\beqs
\Big|\partial^k \|A_N^{r_N^+}(x)\|\Big|\leq \|A_N^{r_N^+}(x)\|^{1+k\eta},\quad 1\leq k\leq l.
\eeqs
Moreover, direct computation from (\ref{ANrN+}) and (\ref{ANrN-}) shows
$$e_N(x)=\big(s_N(x)-u_N(x)\big)-\big(\overline s_N(x)-\overline u_N(x)\big),\quad \forall x\in I_N.$$
Due to the definition of $e_N(x)$, it holds that
$$s_N(x)-u_N(x)=\phi_0(x),\quad \forall x\in \frac{I_N}{10},$$
$$\|s_N(x)-u_N(x)-\phi_0(x)\|_{C^k}<\lambda^{-2+k\eta},\quad \forall x\in I_N\backslash\frac{I_N}{10}.$$
This also implies that for $x\in I_N$ and $1\leq k \leq l$,
$$\|s_N-u_N\|_{C^k}<\|\phi_0\|_{C^k}+\lambda^{-2+k\eta}\leq \lambda^{\eta\cdot q_N}.$$
Until now we have verified all the conditions in Proposition \ref{prop.lowerbound} for the $N$th step.

\subsubsection{The $n$th step.}
Now we assume that $A_N, \cdots, A_{n-1}$ satisfy all the conditions in the proposition.
Define $s_{n-1}(x):={s(A_{n-1}^{r_{n-1}^+}(x))}$ and $u_{n-1}(x):={u(A_{n-1}^{r_{n-1}^-}(x))}$ for $x\in \mathbb S^1$. Then
\beq\label{Anhyperbolic}
  \{A_{n-1}(x),A_{n-1}(Tx),\cdots A_{n-1}(T^{r_{n-1}^+-1}x)\}\ \text{is}\ \lambda_{n-1}-\text{hyperbolic},\quad  \forall x\in I_{n-1};
\eeq
\beq s_{n-1}(x)-u_{n-1}(x)=\phi_0(x),\quad x\in \frac{I_{n-1}}{10};\eeq
\beq
|s_{n-1}(x)-u_{n-1}(x)|\geq \frac12|\phi_0(x)|\geq \left(\frac{b_{n-1}}{20}\right)^{l+1},\quad x\in I_{n-1}\backslash \frac{I_{n-1}}{10}.\eeq
Moreover, from assumption we have that for $x\in \frac{I_{n-1}}{10}$ and $1\leq k \leq l$,
\beq
\Big|\partial^k \|A_{n-1}^{r_{n-1}^+}(x)\|\Big|\leq \|A_{n-1}^{r_{n-1}^+}(x)\|^{1+k\eta};
\eeq
\beq\label{paksnun}
|\partial^k( s_{n-1}(x)-u_{n-1}(x))|\leq  \lambda^{r_{n-1}\cdot\eta}.
\eeq
\begin{lemma}\label{lemma.hyperbolic}
Let $x_0=x,x_1,\cdots, x_m$ be a $T$-orbit with $x_0,x_m\in I_n$ and $x_i\notin I_n$ for $0<i<m$. Then $\{A_{n-1}(x_0),\cdots A_{n-1}(x_m)\}$ is $\lambda_n$-hyperbolic.
Moreover, letting $\overline s_n (x):= {s(A_{n-1}^{r_n^+}(x))}$ and $\overline u_n (x):= {u(A_{n-1}^{-r_n^-}(x))}$, we have for $x\in I_{n}$ and $1\leq k\leq l$,
\beqs
\|\overline s_n-s_{n-1}\|_{C^k},\ \|\overline u_n-u_{n-1}\|_{C^k}\leq \lambda^{-\frac13r_{n-1}},
\eeqs
\beqs
\left|\partial^k\|A_{n-1}^{r_n^+}(x)\|\right|\leq \|A_{n-1}^{r_n^+}(x)\|^{1+k\eta}.
\eeqs
\end{lemma}
\begin{proof}
By assumption, $x_0,x_m\in I_n$ and $x_i\notin I_n$ for $0<i<m$.
This means $m=r_n^+(x)$. By Lemma \ref{lemma.returntimeIn}, we have $m=q_n$ or $q_{n+1}$.
If $m=q_n$, then $r_n^+(x)=r_{n-1}^+(x)=q_n$. It follows from the assumption that $\{A_{n-1}(x_0),\cdots A_{n-1}(x_m)\}$ is $\lambda_n$-hyperbolic and the required estimates are trivial.

Now assume that $m=q_{n+1}$.
Let $0=j_0<j_1<\cdots<j_p=m$ be the returning times to $I_{n-1}$. Then it holds that $j_t-j_{t-1}\geq q_{n-1}$. Moreover, due to the precise analysis on the returning time in the proof of Lemma \ref{lemma.returntimeIn}, there exists $s\approx p/2$ such that
$$j_t-j_{t-1}= q_n,\quad \text{for}\ t\neq s+1,$$
$$ j_{s+1}-j_s=q_{n-1},\quad |T^{j_s}x|\geq \frac{b_{n-1}}4,\quad |T^{j_{s+1}}x|\geq \frac{b_{n-1}}4.$$
By the inductive assumption (\ref{Anhyperbolic}) -- (\ref{paksnun}), we have for $t\neq s+1$,
$$\|A_{n-1}^{j_t-j_{t-1}}(T^{j_{t-1}}x)\|\geq \lambda_{n-1}^{j_t-j_{t-1}}\geq \lambda_{n-1}^{q_n},$$
$$\left|\partial^k\|A_{n-1}^{j_t-j_{t-1}}(T^{j_{t-1}}x)\|\right|\leq \|A_{n-1}^{j_t-j_{t-1}}(T^{j_{t-1}}x)\|^{1+k\eta}.$$
$$\left|s(A_{n-1}^{j_{t+1}-j_t}(T^{j_{t-1}}x)) -u(A_{n-1}^{j_t-j_{t-1}}(T^{j_{t-1}}x))\right| \geq \frac12 |\phi_0(T^{j_{t-1}}x)|\geq b_n^{l+1}.$$
\beqs
\left|\partial^k( s(A_{n-1}^{j_{t+1}-j_t}(T^{j_{t-1}}x)) -u(A_{n-1}^{j_t-j_{t-1}}(T^{j_{t-1}}x))\right|\leq  \lambda^{(j_{t+1}-j_t)\cdot\eta}.
\eeqs
By finite-Liouvillean condition,
$$\lambda_{n-1}^{q_n}\geq \lambda^{(1-\varepsilon)q_n}\gg (100q_{n+1})^{l+1}\geq b_n^{-(l+1)}.$$
Hence we have verified the conditions of Lemma \ref{lemma.basic} and Lemma \ref{lemma.hderivative} to the decomposition
$$A_{n-1}^{j_s}(x)=A_{n-1}^{j_s-j_{s-1}}(T^{j_{s-1}}x)\cdots A_{n-1}^{j_2-j_{1}}(T^{j_{1}}x)A_{n-1}^{j_1}(x).$$
Applying these lemmas inductively, we have
\beq\|A_{n-1}^{j_s}(x)\|\geq \lambda_{n-1}^{j_{s}}\cdot \prod_{1\leq t \leq s-1}\frac12 |\phi_0(T^{j_t}x)|, \label{Ajsnorm}\eeq
\beqs
\left|\partial^k\|A_{n-1}^{j_s}(x)\|\right|\leq \|A_{n-1}^{j_s}(x)\|^{1+k\eta},
\eeqs
\beqs\|u(A_{n-1}^{j_s}(x))-u(A_{n-1}^{j_s-j_{s-1}}(T^{j_{s-1}}x))\|_{C^k} <\lambda^{-(2-k\eta)q_{n+1}}.
\eeqs
Recall the proof of Lemma \ref{lemma.returntimeIn0} that $\{T^{j_t}x\}$ is orderly arranged with a step size $|z_n|$.
Hence $$s\leq \frac{b_{n-1}}{|z_n|}\leq 4\frac{q_{n+1}}{q_n},$$
and we have
$$|T^{j_1}x|\geq b_n\geq \frac14 q_{n+1}^{-1}, $$
$$  |T^{j_2}x|\geq b_n+|z_n|\geq \frac12q_{n+1}^{-1}, $$
$$ |T^{j_3}x|\geq b_n+2|z_n|\geq q_{n+1}^{-1},\quad \cdots,$$
$$ |T^{j_{s-1}}x|\geq b_n+(s-2)|z_n|\geq \frac{s-2}2q_{n+1}^{-1}.$$
This leads to
\beas
\prod_{1\leq t \leq s-1}\frac12 |\phi_0(T^{j_t}x)|
&\geq& \frac14 q_{n+1}^{-1}\cdot \prod_{1\leq t \leq 4q_{n+1}/q_n} \frac12\cdot \left(\frac{t}{2}q_{n+1}^{-1}\right)^{l+1}\\
&\geq& \frac14 q_{n+1}^{-1}\cdot \left(\frac14q_{n+1}^{-1}\right)^{\frac{4(l+1)q_{n+1}}{q_n}} \cdot \left(\left(\frac{4q_{n+1}}{q_n}\right)!\right)^{l+1}.
\eeas
Then we apply Stirling's formula to the righthand side, and thus obtain
\bea\nonumber
\prod_{1\leq t \leq s-1}\frac12 |\phi_0(T^{j_t}x)|
&\geq& \frac14 q_{n+1}^{-1}\cdot \left(\frac{8\pi q_{n+1}}{q_n}\right)^{\frac{l+1}2}\cdot \left(eq_n\right)^{-\frac{4(l+1)q_{n+1}}{q_n}}\\ \label{stirling}
&\geq& \lambda^{-q_{n+1}\cdot S_n},
\eea
where
$$S_n:=\frac1{\log\lambda}\cdot\left(\frac{\log(4q_{n+1})}{q_{n+1}} +\frac{4(l+1)\log(eq_n)}{q_n}\right).$$
Then by (\ref{Ajsnorm}) and (\ref{stirling}), we have
$$\|A_{n-1}^{j_s}(x)\|\geq \lambda_{n-1}^{j_{s}}\cdot \lambda^{-q_{n+1}\cdot S_n}. $$
Similar as above, we could obtain
$$\|A_{n-1}^{m-j_{s+1}}(T^{j_{s+1}}x)\|\geq \lambda_{n-1}^{m-j_{s+1}}\cdot \lambda^{-q_{n+1}\cdot S_n},$$
\beqs
\left|\partial^k\|A_{n-1}^{m-j_{s+1}}(T^{j_{s+1}}x)\|\right|\leq \|A_{n-1}^{m-j_{s+1}}(T^{j_{s+1}}x)\|^{1+k\eta}.
\eeqs
$$\|s(A_{n-1}^{m-j_{s+1}}(T^{j_{s+1}}x))-s(A_{n-1}^{j_{s+2}-j_{s+1}}(T^{j_{s+1}}x))\|_{C^k}<\lambda^{-q_{n+1}(2-k\eta)}.$$
On the other hand, by the inductive assumption (\ref{Anhyperbolic}) -- (\ref{paksnun}), we have
$$\|A_{n-1}^{j_{s+1}-j_{s}}(T^{j_s}x)\|\geq \lambda_{n-1}^{j_{s+1}-j_{s}}=\lambda_{n-1}^{q_{n-1}},$$
\beqs
\left|\partial^k\|A_{n-1}^{j_{s+1}-j_{s}}(T^{j_s}x)\|\right|\leq \|A_{n-1}^{j_{s+1}-j_{s}}(T^{j_s}x)\|^{1+k\eta},
\eeqs
$$|s(A_{n-1}^{j_{s+1}-j_{s}}(T^{j_s}x))-u(A_{n-1}^{j_s-j_{s-1}}(T^{j_{s-1}}x))|\geq \frac12|\phi_0(T^{j_s}x)|\geq \left(\frac{b_{n-1}}{8}\right)^{l+1},$$
$$|s(A_{n-1}^{j_{s+2}-j_{s+1}}(T^{j_{s+1}}x))-u(A_{n-1}^{j_{s+1}-j_{s}}(T^{j_{s}}x))|\geq \frac12|\phi_0(T^{j_{s+1}}x)|\geq \left(\frac{b_{n-1}}{8}\right)^{l+1}.$$
By finite-Liouvillean condition,
$$\lambda_{n-1}^{q_{n-1}}\geq \lambda^{(1-\varepsilon)q_{n-1}}\gg (100q_{n+1})^{l+1}\geq \left(\frac{b_{n-1}}{8}\right)^{-(l+1)}.$$
Now we have verified all the conditions in Lemma \ref{lemma.basic} and Lemma \ref{lemma.hderivative} to the decomposition
$$A_{n-1}^m=A_{n-1}^{m-j_{s+1}}(T^{j_{s+1}}x) A_{n-1}^{j_{s+1}-j_{s}}(T^{j_s}x)A_{n-1}^{j_{s}}(x).$$
Applying these lemmas twice with (\ref{def.lambdan}), we could obtain
$$\|A_{n-1}^{m}(x)\|\geq \lambda_{n-1}^m \lambda^{-2q_{n+1}\cdot T_n}\geq \lambda_n^{m} ,\quad T_n=S_n+\frac1{\log\lambda}\cdot \frac{(l+1)\log(20q_n)}{q_n},$$
which shows that
\begin{center}
  $\{A_{n-1}(x_0),\cdots A_{n-1}(x_m)\}$ is $\lambda_n$-hyperbolic,\quad  $\forall x_0,x_m\in I_n$,
\end{center}
And moreover, we have for $x\in I_n$
\beqs
\|\overline s_n-s_{n-1}\|_{C^k},\ \|\overline u_n-u_{n-1}\|_{C^k}\leq \lambda^{-\frac13r_{n-1}},
\eeqs
\beqs
\left|\partial^k\|A_{n-1}^{r_n^+}(x)\|\right|\leq \|A_{n-1}^{r_n^+}(x)\|^{1+k\eta}.
\eeqs
\end{proof}

Define $e_n(x)$ to be the following $2\pi$-periodic function:
$$
e_n(x)=\left\{
         \begin{array}{ll}
           (s_{n-1}-u_{n-1})(x)-(\overline s_n-\overline u_n)(x), & \hbox{$x\in \frac{I_n}{10}$;} \\
           h_n^\pm(x), & \hbox{$x\in I_n\backslash \frac{I_n}{10}$;} \\
           0, & \hbox{$x\in \mathbb S^1\backslash I_n$,}
         \end{array}
       \right.
$$
where $h_n^\pm(x)$ is a $C^l$-function satisfies
$$\frac{d^jh_n^\pm}{dx^j}\left(\pm\frac{b_n}{10}\right)= \frac{d^j(s_{n-1}-u_{n-1})}{dx^j}\left(\pm\frac{b_n}{10}\right) +\frac{d^j(\overline s_n-\overline u_n)}{dx^j}\left(\pm\frac{b_n}{10}\right),$$
$$\frac{d^jh_n^\pm}{dx^j}\left(\pm b_n\right)= 0,\quad i=1,2,\quad j=0,1,2,\cdots,l.$$
From Lemma \ref{lemma.hyperbolic}, we have for $x\in \frac{I_n}{10}$ and $1\leq k\leq l$, $\|e_n(x)\|_{C^k}\leq \lambda^{-\frac13r_{n-1}}$.
By Cramer's rules, we have
$$\|h_n^\pm(x)\|_{C^k}\leq \lambda^{-\frac13r_{n-1}},\quad 1\leq k\leq l.$$
Hence it holds that for $x\in \mathbb S^1$ and $1\leq k\leq l$,
\beq\label{ennorm}\|e_n(x)\|_{C^k}\leq \lambda^{-\frac13r_{n-1}}.\eeq

Now define
$$\phi_n(x):=\phi_{n-1}(x)+e_n(x),\quad A_n(x):=\Lambda\cdot R_{\frac{\pi}{2}-\phi_n(x)},$$
$$s_n(x):={s(A_n^{r_n^+}(x))},\quad u_n(x):={s(A_n^{-r_n^-}(x))}.$$
Similarly as in the $N$th step, we could easily verify that for $x\in I_n$,
$$A_n^{r_n^+}(x)=A_{n-1}^{r_n^+}(x)\cdot R_{-e_n(x)},\quad A_n^{-r_n^-}(x)=R_{-e_n(T^{-r_n^-}x)}\cdot  A_{n-1}^{-r_n^-}(x),$$
\beq\label{ensnun}e_n(x)=\left(s_{n}(x)-u_{n}(x)\right)-\left(\overline s_n(x)-\overline u_n(x)\right).\eeq
Then combined with Lemma \ref{lemma.hyperbolic} and the fact that rotations do not affect the norms of the matrices, we have
\beqs
 \{A_{n}(x),A_{n}(Tx),\cdots A_{n}(T^{r_{n}^+-1}x)\}\  \text{is}\ \lambda_{n}-\text{hyperbolic},\quad x\in I_n;
\eeqs
\beqs
\left|\partial^k\|A_{n}^{r_n^+}(x)\|\right|\leq \|A_{n}^{r_n^+}(x)\|^{1+k\eta},\quad \left|\partial^k\|A_{n}^{r_n^-}(x)\|\right|\leq \|A_{n}^{r_n^-}(x)\|^{1+k\eta}.
\eeqs
Moreover, by (\ref{ennorm}), (\ref{ensnun}) and Lemma \ref{lemma.hyperbolic} we have
\begin{align*}
\left|\partial^k\left(s_{n}(x)-u_{n}(x)\right)\right|&\leq \left|\partial^k\left(s_{n-1}(x)-u_{n-1}(x)\right)\right| +O(\lambda_n^{-\frac13r_{n-1}})\\
&\leq |\partial^k \phi_0|+\sum_{n\geq N}O(\lambda_n^{-\frac13r_{n-1}})\leq 2\| \phi_0\|_{C^k}.
\end{align*}
And thus
$$\left\|s_{n}(x)-u_{n}(x)\right\|_{C^k}\leq 2\|\phi_0\|_{C^k}\ll\lambda^{q_n\cdot\eta}.$$
Furthermore, we could easily verify that for $x\in I_{n}$,
$$\left\|\phi_n(x)-\phi_{n-1}(x)\right\|_{C^l}\leq \lambda_n^{-\frac13r_{n-1}}.$$
$$s_{n}(x)-u_{n}(x)=\phi_0(x),\quad x\in \frac{I_{n}}{10},$$
$$|s_{n}(x)-u_{n}(x)|\geq \frac12|\phi_0(x)|\geq \left(\frac{b_{n}}{20}\right)^{l+1},\quad I_{n}\backslash \frac{I_{n}}{10}.$$
Until now we have obtained all the conclusions in Proposition \ref{prop.lowerbound} for the $n$th step.

\subsection{Proof of Proposition \ref{prop.upperbound}.}
Define $e_n(x)$ to be the following $2\pi$-periodic function:
$$
\widetilde e_n(x)=\left\{
         \begin{array}{ll}
           (s_n-u_n)(x), & \hbox{$x\in \frac{I_n}{10}$;} \\
           \widetilde h_n^\pm(x), & \hbox{$x\in I_n\backslash \frac{I_n}{10}$;} \\
           0, & \hbox{$x\in \mathbb S^1\backslash I_n$,}
         \end{array}
       \right.
$$
where $\widetilde h_n^\pm(x)$ is a $C^l$-function satisfies
$$\frac{d^j \widetilde h_n^\pm}{dx^j}\left(\pm\frac{b_n}{10}\right)= \frac{d^j(s_n-u_n)}{dx^j}\left(\pm\frac{b_n}{10}\right),$$
$$\frac{d^j\widetilde h_N^\pm}{dx^j}\left(\pm b_n\right)= 0,\quad i=1,2,\quad j=0,1,2,\cdots,l.$$
From Proposition \ref{prop.lowerbound}, for $x\in \frac{I_n}{10}$ and $0\leq k\leq l$, $\|s_n-u_n\|_{C^k}=O(q_{n+1}^{-(l+1-k)})$. Hence from Cramer's rule we have that $\|\widetilde h_n^\pm\|_{C^k}=O(q_{n+1}^{-1})$.
This shows that for $x\in \mathbb S^1$,
$$\|\widetilde e_n^\pm\|_{C^k}=O(q_{n+1}^{-1}).$$

Define $$\widetilde \phi_n(x):=\phi_n(x)+\widetilde e_n(x),\quad \widetilde A_n(x):=\Lambda\cdot R_{\frac\pi2-\widetilde \phi_n(x)}.$$
Hence it is clear that $\widetilde A_n\rightarrow D_l$ in $C^l$ topology.
Moreover, since $\widetilde \phi_n(x)=\phi_n(x)$ on $\mathbb S^1\backslash I_n$ and for each $x\in I_n$, $\{A_n(x),A_n(Tx),\cdots, A_n(T^{r_n^+-1}x)\}$ is $\lambda_n$-hyperbolic (obtained by Proposition \ref{prop.lowerbound}), we have
\begin{center}
  $\{\widetilde A_{n}(x),\widetilde A_{n}(Tx),\cdots,\widetilde A_{n}(T^{r_{n}^+-1}x)\}$ is $\lambda_{n}$-hyperbolic,\quad $x\in I_n$.
\end{center}
Define $$\widetilde s_n(x)={s(\widetilde A_n^{r_n^+}(x))},\quad \widetilde u_n(x)={u(\widetilde A_n^{-r_n^-}(x))}.$$
Similar to (\ref{ensnun}), it is also direct to verify that $$\widetilde s_n(x)-\widetilde u_n(x)=s_n(x)-u_n(s)-\widetilde e_n(x).$$ Then it holds that for $x\in \frac{I_n}{10}$,
$$\widetilde s_n(x)=\widetilde u_n(x).$$
Until now we have obtained all the conclusions in Proposition \ref{prop.upperbound}.

\section{Proof of Theorem \ref{thm.main2}: Discontinuity of LE in $C^\infty$ topology}
In this section, we prove discontinuity of LE in $C^\infty$ topology. Compared with \cite{WY13}, we define the sample function as $\phi_0(x)\approx\exp\big(-\big(\log\frac{1}{|x|}\big)^{\sigma}\big)$ and proceed Young's inductive argument with Brjuno frequency.
Our estimates are more technical and sharp due to the weaker condition on frequency.

We first introduce the topology in $C^\infty(\mathbb S^1)$. For any smooth function $f$ defined on $\mathbb S^1$, let
\beqs
\|f\|_{C^k}:=\max_{j\in\{0,1,2,\cdots,k\}}\max_{x\in \mathbb S^1}|\partial^kf(x)|,\quad k\in \mathbb N.
\eeqs
It is clear that $C^{\infty}(\mathbb S^1)$ is a complete metric space with the metric $d(f,g):=\sum_{k\in \mathbb N}\frac{2^{-k}\cdot\|f-g\|_{C^k}}{\|f-g\|_{C^k}+1}$. Moreover, $f_n\rightarrow f$ in $C^{\infty}(\mathbb S^1)$ topology if and only if $\|f_n-f\|_{C^k}\rightarrow 0$ as $n\rightarrow \infty$ for any $k\in \mathbb N$.

We now give the following settings.
\begin{itemize}
  \item The frequency: Let $\alpha$ be a Brjuno number, i.e. there exists $0<\delta<1$ such that
        \beqs\beta_\delta:=\limsup_{n\rightarrow\infty} \frac{\log q_{n+1}}{q_n^\delta}<\infty.\eeqs
  \item Critical interval $I_n$ and first returning time $r_n^{\pm}$ are defined as in Section 4.
  \item The sample function: Let $\sigma>1$ be a constant satisfying $\sigma\delta^\frac12<1$. The $2\pi$-periodic smooth function $\phi_0\in C^{\infty}(\mathbb S^1)$ is defined as for $x\in [0,2\pi)$,
     $$\phi_0(x)=\exp\Big(-\big(\log\frac{8\pi}{x}\big)^{\sigma} -\big(\log\frac{8\pi}{2\pi-x}\big)^{\sigma}\Big).$$
\end{itemize}
Fix $\varepsilon>0$ small.
Let $N$ and $\lambda$ be sufficiently large such that
\beq\label{choiceN3}\sum_{n\geq N}q_{n-1}^{-(1-\delta^{\frac12})}<\frac{\varepsilon}{10}, \quad  \sum_{n>N} q_{n+1}^{-\frac{1}2} \leq\frac\varepsilon4, \quad
\lambda>\max\Big\{e^{(2\beta_\delta)^\sigma} ,e^{q_N^{q_N}}\Big\}.\eeq
Moreover, we inductively define $\lambda_n$ by
$$\lambda_{N-1}:=\lambda,\quad \log\lambda_n:=\log\lambda_{n-1}-10q_{n-1}^{-(1-\delta^{\frac12})}.$$
As $n\rightarrow \infty$ the sequence $\lambda_n$ decreases to some $\lambda_\infty$, which satisfies $\lambda^{1-\frac\varepsilon{10}}<\lambda_\infty<\lambda$.
\begin{prop}\label{lemma.Ilowerbound}
There exist functions $\phi_n(x)$ on $\mathbb S^1$ ($n=N,N+1,\cdots$) such that
\begin{description}
  \item[$(1)_n$] The function $\phi_n(x)$ is exactly the same as $\phi_{n-1}(x)$ for $x\notin I_n$, and moreover it holds that for $1\leq k\leq \log n$, $$\left\|\phi_n(x)-\phi_{n-1}(x)\right\|_{C^k}\leq \lambda_n^{-\frac1{100}r_{n-1}}.$$
  \item[$(2)_n$] The sequence $\{A_n(x), A_n(Tx),\cdots,A_n(T^{r_n^+-1}x)\}$ is $\lambda_n$-hyperbolic for $x\in I_n$, where $A_n(x):=\Lambda R_{\frac\pi2-\phi_n(x)}$.
  \item[$(3)_n$] Let $s_n(x)={s(A_n^{r_n^+}(x))}$ and $u_n(x)={s(A_n^{-r_n^-}(x))}$. Then it holds that
\begin{description}
  \item[$(3.1)_n$] $s_n(x)-u_n(x)=\phi_0(x)$, for $x\in \frac{I_n}{10}$;
  \item[$(3.2)_n$] $|s_n(x)-u_n(x)|\geq \frac12 |\phi_0(x)|\geq e^{-(-\log (b_n/20))^{\sigma}}$, for $x\in I_n\backslash \frac{I_n}{10}$.
\end{description}
\end{description}
\end{prop}

\begin{prop}\label{lemma.Iupperbound}
There exists $\lambda_0=\lambda_0(\alpha,\phi_0)$ such that for $\lambda>\lambda_0$ the following holds.
There exist functions $\widetilde\phi_n(x)$ on $\mathbb S^1$ ($n=N,N+1,\cdots$) such that
\begin{description}
  \item[$(1')_n$] The function $\widetilde\phi_{n}$ converges to $\phi_n$ in $C^\infty$ topology, i.e. for $k$ satisfying $1\leq k\leq \log n$,  $$\big\|\phi_n(x)-\widetilde\phi_{n}(x)\big\|_{C^k}\leq q_{n+1}^{-2}.$$
  \item[$(2')_n$] The sequence $\{\widetilde A_n(x), \widetilde A_n(Tx),\cdots,\widetilde A_n(T^{r_n^+-1}x)\}$ is $\lambda_n$-hyperbolic for $x\in I_n$, where $\widetilde A_n(x):=\Lambda R_{\frac\pi2-\widetilde \phi_n(x)}$.
  \item[$(3')_n$] Let $\widetilde s_n(x)={s(\widetilde A_n^{r_n^+}(x))}$ and $\widetilde u_n(x)={s(\widetilde A_n^{-r_n^-}(x))}$. Then it holds that
      $$\widetilde s_n(x)=\widetilde u_n(x),\quad x\in \frac{I_n}{10}.$$
\end{description}
\end{prop}

\begin{remark}
The proofs of Proposition \ref{lemma.Ilowerbound} and \ref{lemma.Iupperbound} are given in Section 8.
\end{remark}

Once the above two propositions established, we could adjust the parameters in Section 4 to prove the discontinuity of LE for $\{\widetilde A_n\}$.
Define $ B_n$ by \beqs
B_{n}=\bigcup_{l=1}^{[q_{n+1}^{3/2}]}\left(B(0,q_{n+1}^{-2})-l\alpha\right),\quad n\geq N.
\eeqs
By (\ref{choiceN3}), it is straightforward to compute that for $N$ large enough,
$$ \text{meas}\bigg(\bigcup_{n\geq N}B_{n}\bigg)\leq \sum_{n>N} q_{n+1}^{-\frac{1}2} \leq\frac\varepsilon4.$$

Let $x\notin \bigcup_{n\geq N}B_{n}$, $m\geq n$ and $i:=[q_{n+1}^{\frac 32}]$.
Consider the trajectory $\{T^jx:0\leq j\leq i\}$, let $i_0=i_0(x)\geq 0$ be the first returning time to $I_{n}$, and $i_1=i_1(x)\geq 0$ be the last returning time to $I_{n}$.
By Lemma \ref{liouville}, we have that $$i_0,i-i_1\leq q_{n+1}\ll i.$$
By the definition of $B_n$, we know that the trajectory $\{T^jx:0\leq j\leq i\}$ has not entered $B(0,q_{n+1}^{-2})$, but has entered $I_{n}$.
Let $i_0=j_0<j_1<\cdots<j_p=i_1$ be the returning times of $x\in I_{n-1}$.
Now consider the decomposition
$$A_m^i(x)=A_m^{i-i_1}(T^{i_1}x)\cdot A_m^{i_1-i_0}(T^{i_0}x) \cdot A_m^{i_0}(x),$$
\beqs A_m^{i_1-i_0}(T^{i_0}x)=A_m^{j_{p}-j_{p-1}}(T^{j_{p-1}}x)\cdots A_m^{j_{2}-j_{1}}(T^{j_{1}}x)A_m^{j_{1}-j_{0}}(T^{j_{0}}x).\eeqs
As before, we are going to repeatedly apply Lemma \ref{lemma.basic} to obtain the lower bound of the norms. This requires us to verify that the angles between matrices are larger than the errors of size $\|A_m^{j_{l+1}-j_l}(T^{j_{l}}x)\|^{-1}$.
Note that from Lemma \ref{lemma.returntimeIn}, either
$$j_{l+1}-j_l= q_n,\quad |T^{j_l}x|\geq q_{n+1}^{-2},\quad \text{or}$$
$$ j_{l+1}-j_l=q_{n-1},\quad |T^{j_l}x|\geq \frac{b_{n-1}}4,\quad |T^{j_{l+1}}x|\geq \frac{b_{n-1}}4$$
hold.
By $(3)_n$ in Proposition \ref{lemma.Ilowerbound} and definition of $\phi_0$,
$$
|s_n(T^{j_l}x)-u_n(T^{j_l}x)|>\frac12|\phi_0(T^{j_l}x)|\geq \left\{
                                   \begin{array}{ll}
                                                                   e^{-(2\log q_{n+1})^\sigma}, & \hbox{if $j_{l+1}-j_{l}=q_{n}$;} \\
                                                                   e^{-(\log(8 q_{n}))^\sigma}, & \hbox{if $j_{l+1}-j_{l}=q_{n-1}$.}
                                                                 \end{array}
                                                               \right.
$$
To verify the condition in Lemma \ref{lemma.basic}, we require
\beqs e^{-(2\log q_{n+1})^\sigma}\geq \lambda^{-\eta q_{n}}\gg\lambda^{-q_n},\quad e^{-(\log(8 q_{n}))^\sigma}\geq  \lambda^{-\eta q_{n-1}}\gg\lambda^{-q_{n-1}},\quad \eta=q_{n-1}^{-(1-\delta^{\frac12})},\eeqs
i.e.
\beq\label{condition1}
\frac{(2\log q_{n+1})^\sigma}{q_n\cdot q_{n-1}^{-(1-\delta^{\frac12})}}\leq \log \lambda,\quad
\frac{(\log 8q_{n})^\sigma}{q_{n-1}^{\delta^{\frac12}}}\leq \log \lambda.\eeq
Recall that the frequency $\alpha$ is a Brjuno number,
which means that for all large $n$, $\log q_{n+1}\leq \beta_\delta\cdot q_n^\delta$.
Since $\delta^\frac12\cdot \sigma<1$, we have $\frac{(2\log q_{n+1})^\sigma}{q_n^{\delta^\frac12}}\leq (2\beta_\delta)^\sigma$ for all large $n$.
Thus if $N$ is sufficiently large, $n\geq N$ and
$$\lambda \geq e^{(2\beta_\delta)^\sigma},$$
then (\ref{condition1}) is satisfied and hence Lemma \ref{lemma.basic} is applicable. By the same argument in Section 4.1, we could obtain
\beqs
\frac1i\log \|A_m^i(x)\|\geq (1-\frac\varepsilon4)
\log\lambda, \quad i=[q_{n+1}^{\frac 32}],\quad x\notin B_{n},\quad m\geq n.\eeqs
Finally, we could conclude that $\|A_n-D_\infty\|_{C^k}\rightarrow 0$ as $( e^k \leq ) \ n\rightarrow \infty$ and
\beq\label{Ilowerbound}L(D_\infty)\geq (1-\varepsilon)\log\lambda.\eeq

Once Proposition \ref{lemma.Iupperbound} is established, we could follow the same argument in Section 4.2 to prove that $\|\widetilde A_n-D_\infty\|_{C^k}\rightarrow 0$ as $( e^k \leq ) \ n\rightarrow \infty$ and
\beq\label{Iupperbound}L(\widetilde A_n)\leq (1-\delta)\log \lambda,\quad \delta=\frac1{300}.\eeq
Combining (\ref{Ilowerbound}) and (\ref{Iupperbound}), we have proved that the LE is discontinuous at $D_\infty$ in $C^\infty$ topology.

\section{Proof of Theorem \ref{thm.main5} and \ref{thm.main3}: Discontinuity of LE in $G^s$ topology}
In this section, we prove discontinuity of LE in $G^s$ topology with Diophantine frequency and strong Diophantine frequency.
Compared with \cite{GWYZ}, our proofs are more technical and sharp due to the weaker condition on frequency.

For any smooth function $f$ defined on $\mathbb S^1$, let
\beqs
\|f\|_{s,K}:=\frac{4\pi^2}{3} \sup_{k\in\mathbb N}\frac{(1+k)^2}{K^k(k!)^s}\|\partial^kf\|_{C^0},\quad s>1,\quad K>0,
\eeqs
\beqs
G^{s,K}(\mathbb S^1)=\Big\{f\in C^{\infty}(\mathbb S^1,\mathbb R): \|f\|_{s,K}<\infty\Big\},\quad G^s(\mathbb S^1)=\bigcup_{K>0}G^{s,K}(\mathbb S^1).
\eeqs
It is clear that $G^{s,K}(\mathbb S^1)$ is a Banach space. Moreover, $G^s(\mathbb S^1)$ is equipped with the usual inductive limit topology, i.e. $f_n\rightarrow f$ in $G^s(\mathbb S^1)$ topology if and only if $\|f_n-f\|_{s,K}\rightarrow 0$ as $n\rightarrow \infty$ for some $K>0$.
\subsection{Diophantine frequency and $G^s$ ($s>\tau+1$) topology}
We are going to prove Theorem \ref{thm.main3}.
We first give the following settings.
\begin{itemize}
  \item The frequency: Let $\alpha\in \text{DC}(\gamma, \tau)$ with $\gamma>0$ and $\tau>1$, i.e.
        \beqs q_{n+1}\leq \gamma^{-1} q_n^\tau,\quad  \forall n\in \mathbb N.\eeqs
        We also assume that $\alpha$ is not a bounded type number.
  \item Critical interval $I_n$ and first returning time $r_{n}^\pm$ are defined as in Section 4.
  \item The sample function: Let $s>\tau+1$. The $2\pi$-periodic smooth function $\phi_0\in G^{s}(\mathbb S^1)$ is defined as for $x\in [0,2\pi)$,
     $$\phi_0(x)=\exp\left(-x^{\frac{-1}{s-1}} -(2\pi-x)^{\frac{-1}{s-1}}\right).$$
\end{itemize}
Fix small $\varepsilon>0$. Let $N$ and $\lambda$ be sufficiently large such that
\beq\label{choiceN}\sum_{n\geq N}q_n^{\frac{\tau}{s-1}-1}<\frac{\varepsilon}{10^4}, \quad \sum_{n\geq N}q_n^{\frac12(1-\frac{s-1}{\tau})}<\frac{\varepsilon}{10^4},\eeq
\beq\label{choicelambda}\lambda>\max\Big\{\exp\Big(\frac{10^4}{\varepsilon}\gamma^{-\tau^{-1}}\Big) ,e^{q_N^{q_N}}\Big\}.\eeq
Moreover, we inductively define $\lambda_n$ and $\widetilde \lambda_n$ by
$$\lambda_{N-1}=\widetilde\lambda_{N-1}:=\lambda,$$
\beqs
\log\lambda_n:=\log\lambda_{n-1}-2\cdot10^3q_{n-1}^{\frac{\tau}{s-1}-1}, \quad\log\widetilde\lambda_n:=\log\lambda_{n-1}+2\cdot10^3q_{n-1}^{\frac{\tau}{s-1}-1}.
\eeqs
Since $\frac{\tau}{s-1}-1<0$, as $n\rightarrow \infty$ the sequences $\lambda_n$ and $\widetilde\lambda_n$ decrease respectively to some $\lambda_\infty$ and $\widetilde\lambda_\infty$, which satisfy $\lambda^{1-\frac\varepsilon{10}}<\lambda_\infty<\widetilde\lambda_\infty<\lambda^{1+\frac\varepsilon{10}}$.
\begin{prop}\label{lemma.Glowerbound}
There exist functions $\phi_n(x)$ on $\mathbb S^1$ ($n=N,N+1,\cdots$) such that
\begin{description}
  \item[$(1)_n$] The function $\phi_n(x)$ is exactly the same as $\phi_{n-1}(x)$ for $x\notin I_n$, and moreover it holds that for some $K>0$, $$\big\|\phi_n(x)-\phi_{n-1}(x)\big\|_{s,K}\leq \lambda_n^{-\frac1{100}q_{n-1}},\quad n>N.$$
  \item[$(2)_n$] The sequence $\{A_n(x), A_n(Tx),\cdots,A_n(T^{r_n^+-1}x)\}$ is $\lambda_n$-hyperbolic for $x\in I_n$, where $A_n(x):=\Lambda R_{\frac\pi2-\phi_n(x)}$. Moreover, $$\Big\| \|A_n^{r_n^\pm}(x)\|\Big\|_{s,K}\leq \widetilde \lambda_n^{r_n^\pm}, \quad \Big\| \|A_n^{r_n^\pm}(x)\|^{-1}\Big\|_{s,K}\leq \lambda_n^{-r_n^\pm}.$$
  \item[$(3)_n$] Let $s_n(x)={s(A_n^{r_n^+}(x))}$ and $u_n(x)={s(A_n^{-r_n^-}(x))}$. Then it holds that
\begin{description}
  \item[$(3.1)_n$] $s_n(x)-u_n(x)=\phi_0(x)$, for $x\in \frac{I_n}{10}$;
  \item[$(3.2)_n$] $|s_n(x)-u_n(x)|\geq \frac12 |\phi_0(x)|\geq e^{-(b_n/20)^{\frac{-1}{s-1}}}$, for $x\in I_n\backslash \frac{I_n}{10}$.
\end{description}
\end{description}
\end{prop}

\begin{prop}\label{lemma.Gupperbound}
There exists $\lambda_0=\lambda_0(\alpha,\phi_0)$ such that for $\lambda>\lambda_0$ the following holds.
There exist functions $\widetilde\phi_n(x)$ on $\mathbb S^1$ ($n=N,N+1,\cdots$) such that
\begin{description}
  \item[$(1')_n$] The function $\widetilde\phi_{n}$ converges to $\phi_n$ in $G^s$ topology, i.e. for some $K>0$, $$\big\|\phi_n(x)-\widetilde\phi_{n}(x)\big\|_{s,K}\leq q_{n+1}^{-2}.$$
  \item[$(2')_n$] The sequence $\{\widetilde A_n(x), \widetilde A_n(Tx),\cdots,\widetilde A_n(T^{r_n^+-1}x)\}$ is $\lambda_n$-hyperbolic for $x\in I_n$, where $\widetilde A_n(x):=\Lambda R_{\frac\pi2-\widetilde \phi_n(x)}$.
  \item[$(3')_n$] Let $\widetilde s_n(x)={s(\widetilde A_n^{r_n^+}(x))}$ and $\widetilde u_n(x)={s(\widetilde A_n^{-r_n^-}(x))}$. Then it holds that
      $$\widetilde s_n(x)=\widetilde u_n(x),\quad x\in \frac{I_n}{10}.$$
\end{description}
\end{prop}
\begin{remark}
The proofs of Proposition \ref{lemma.Glowerbound} and \ref{lemma.Gupperbound} could be regarded as an improved version of Proposition 3.1 and 3.2 in \cite{GWYZ}. We give the sketches of their proofs in Section 9.
\end{remark}

Once the above two propositions established, we could adjust the parameters in Section 4 to prove the discontinuity of LE for $\{\widetilde A_n\}$.
Define $B_n$  by \beqs
B_{n}=\bigcup_{l=1}^{ [q_{n+1}^{1+a/2}]}\left(B(0,q_{n+1}^{-a})-l\alpha\right),\quad a=\frac{s-1}{\tau},\quad n\geq N.
\eeqs
Note $a>1$. By (\ref{choiceN}), it is straightforward to compute that for $N$ large enough,
$$ \text{meas}\bigg(\bigcup_{n\geq N} B_{n}\bigg)\leq \sum_{n>N} q_{n+1}^{-\frac{a-1}2} \leq\frac\varepsilon4.$$

Let $x\notin \bigcup_{n\geq N}B_{n}$, $m\geq n$ and $i:=[q_{n+1}^{1+\frac a2}]$.
Consider the trajectory $\{T^jx:0\leq j\leq i\}$, let $i_0=i_0(x)\geq 0$ be the first returning time to $ I_{n}$, and $i_1=i_1(x)\geq 0$ be the last returning time to $ I_{n}$.
By Lemma \ref{liouville}, we have that $$i_0,i-i_1\leq q_{n+1}\ll i.$$
By the definition of $B_n$, we know that the trajectory $\{T^jx:0\leq j\leq i\}$ has not entered $B(0,q_{n+1}^{-a})$, but has entered $I_{n}$.
Let $i_0=j_0<j_1<\cdots<j_p=i_1$ be the returning times to $I_{n-1}$.
Now consider the decomposition
$$A_m^i(x)=A_m^{i-i_1}(T^{i_1}x)\cdot A_m^{i_1-i_0}(T^{i_0}x) \cdot A_m^{i_0}(x),$$
\beqs A_m^{i_1-i_0}(T^{i_0}x)=A_m^{j_{p}-j_{p-1}}(T^{j_{p-1}}x)\cdots A_m^{j_{2}-j_{1}}(T^{j_{1}}x)A_m^{j_{1}-j_{0}}(T^{j_{0}}x).\eeqs
As before, we are going to repeatedly apply Lemma \ref{lemma.basic} to obtain the lower bound of the norms. This requires us to verify that the angles between matrices are larger than the errors of size $\|A_m^{j_{l+1}-j_l}(T^{j_{l}}x)\|^{-1}$.
Note that from Lemma \ref{lemma.returntimeIn}, either
$$j_{l+1}-j_l= q_n,\quad |T^{j_l}x|\geq q_{n+1}^{-a},\quad \text{or}$$
$$ j_{l+1}-j_l=q_{n-1},\quad |T^{j_l}x|\geq \frac{b_{n-1}}4,\quad |T^{j_{l+1}}x|\geq \frac{b_{n-1}}4$$
hold.
By $(3)_n$ in Proposition \ref{lemma.Glowerbound} and definition of $\phi_0$,
$$
|s_n(T^{j_l}x)-u_n(T^{j_l}x)|>\frac12|\phi_0(T^{j_l}x)|\geq \left\{
                                   \begin{array}{ll}
                                                                   e^{-10^2q_{n+1}^{\frac a{s-1}}}, & \hbox{if $j_{l+1}-j_{l}=q_{n}$;} \\
                                                                   e^{-10^2q_{n}^{\frac1{s-1}}}, & \hbox{if $j_{l+1}-j_{l}=q_{n-1}$.}
                                                                 \end{array}
                                                               \right.
$$
Hence we require
\beqs e^{-10^2q_{n+1}^{\frac a{s-1}}}\gtrsim \lambda^{-\frac\varepsilon{10}q_{n}}\gg\lambda^{-q_n},\quad e^{-10^2q_{n}^{\frac1{s-1}}}\gtrsim  \lambda^{-\frac\varepsilon{10}q_{n-1}}\gg\lambda^{-q_{n-1}},\eeqs
i.e.
\beq\label{condition}q_{n+1}\leq \Big(\frac{\varepsilon\log \lambda}{10^3}\Big)^{\frac{s-1}a} q_n^{\frac{s-1}a},\quad q_{n}\leq \Big(\frac{\varepsilon\log \lambda}{10^3}\Big)^{s-1} q_{n-1}^{s-1}.\eeq
Recall the frequency $\alpha$ is a Diophantine number, $q_{n+1}\leq \gamma^{-1} q_n^{\tau}$ for all $n\in \mathbb Z_+$.
If $$\lambda \geq \exp\Big(\frac{10^3}{\varepsilon}\gamma^{\frac{-a}{s-1}}\Big),\quad s> 1+\tau,$$ then (\ref{condition}) is satisfied and hence Lemma \ref{lemma.basic} is applicable. By the same argument in Section 4.1, we could obtain
\beqs
\frac1i\log \|A_m^i(x)\|\geq (1-\frac\varepsilon4)
\log\lambda, \quad i=[q_{n+1}^{1+\frac a2}],\quad x\notin B_{n},\quad m\geq n.\eeqs
Finally, we could conclude that there exists a $D_s\in G^s(\mathbb S^1,\mathrm{SL}(2,\mathbb R))$ such that $\|A_n-D_s\|_{s,K}\rightarrow 0$ and
\beq\label{Glowerbound}L(D_s)\geq (1-\varepsilon)\log\lambda.\eeq

Once Proposition \ref{lemma.Gupperbound} is established, we could follow exactly the same process in Section 4.2 to prove that $\|\widetilde A_n-D_s\|_{s,K}\rightarrow 0$ and
\beq\label{Gupperbound}L(\widetilde A_n)\leq (1-\delta)\log \lambda,\quad \delta=\frac1{300}.\eeq
Combining (\ref{Glowerbound}) and (\ref{Gupperbound}), we have proved that the LE is discontinuous at $D_s$ in $G^s$ topology.

\subsection{Strong Diophantine frequency and $G^s$ ($s>2$) topology}
The continuous part of Theorem \ref{thm.main5} has been proved by Theorem 6.3 in \cite{CGYZ}. Here we only give the proof of the discontinuous part, which follows the same process as in Subsection 7.1.
Assume the frequency $\alpha\in \mathrm{SDC}(\gamma, \tau)$ with $\gamma>0$ and $\tau>0$, i.e.
\beqs q_{n+1}\leq \gamma^{-1} q_n (\log q_n)^\tau,\quad  \forall n\in \mathbb N.\eeqs
We also assume that $\alpha$ is not a bounded type number.
Moreover, we define the critical set, critical interval, first returning time and sample function ($s>2$) as the previous subsection.

To determine the largeness of $N$ and $\lambda$, we revise (\ref{choiceN}) and (\ref{choicelambda}) as
\beq\label{choiceN2}
\sum_{n\geq N}q_n^{\frac{2-s}{2(s-1)}}<\frac{\varepsilon}{10^4}, \quad \sum_{n\geq N}q_n^{\frac{2-s}{2}}<\frac{\varepsilon}{10^4}, \quad \gamma^{-1} q_n (\log q_n)^\tau <q_n^{\frac s2},\quad \forall n\geq N,
\eeq
\beqs
\lambda>\max\Big\{\exp\big(10^3\varepsilon^{-1},e^{q_N^{q_N}}\big)\Big\}.
\eeqs
And we redefine $\lambda_n$ and $\widetilde \lambda_n$ by
$$\log\lambda_n:=\log\lambda_{n-1}-2\cdot10^3q_{n-1}^{\frac{2-s}{2(s-1)}}, \quad\log\widetilde\lambda_n:=\log\lambda_{n-1}+2\cdot10^3q_{n-1}^{\frac{2-s}{2(s-1)}}.$$
With these settings, we could prove Proposition \ref{lemma.Glowerbound} and \ref{lemma.Gupperbound} for strong Diophantine frequency and $s>2$.
Now we redefine
\beqs
B_{n}=\bigcup_{l=1}^{ [q_{n+1}^{1+a/2}]}\left(B(0,q_{n+1}^{-a})-l\alpha\right),\quad a=\frac{2(s-1)}{s},\quad n\geq N.
\eeqs
Let $i:=[q_{n+1}^{1+\frac a2}]$.
For $x\notin \bigcup_{n\geq N}B_{n}$ and $m\geq n$, we consider
$$A_m^i(x)=A_m^{i-i_1}(T^{i_1}x)\cdot A_m^{i_1-i_0}(T^{i_0}x) \cdot A_m^{i_0}(x),$$
\beqs A_m^{i_1-i_0}(T^{i_0}x)=A_m^{j_{p}-j_{p-1}}(T^{j_{p-1}}x)\cdots A_m^{j_{2}-j_{1}}(T^{j_{1}}x)A_m^{j_{1}-j_{0}}(T^{j_{0}}x).\eeqs
As before, we will apply Lemma \ref{lemma.basic} to obtain the lower bound of $\|A_m^{i_1-i_0}(T^{i_0}x)\|$ and $\|A_m^i(x)\|$.
To apply Lemma \ref{lemma.basic}, we should verify the condition that the angles are larger than the reciprocal of the norms, which requires
$$q_{n+1}\leq \Big(\frac{\varepsilon\log \lambda}{10^3}\Big)^{\frac{s}2} q_n^{\frac{s}2},\quad q_{n}\leq \Big(\frac{\varepsilon\log \lambda}{10^3}\Big)^{s-1} q_{n-1}^{s-1}.$$
instead of (\ref{condition}). This is satisfied by strong Diophantine condition and (\ref{choiceN2}). Finally, we could conclude that there exists a $D_s\in G^s(\mathbb S^1,\mathrm{SL}(2,\mathbb R))$ such that $\|A_n-D_s\|_{s,K}\rightarrow 0$ and
\beqs L(D_s)\geq (1-\varepsilon)\log\lambda.\eeqs
On the other hand, by Proposition \ref{lemma.Gupperbound}, we could prove that $\|\widetilde A_n-D_s\|_{s,K}\rightarrow 0$ and
\beqs L(\widetilde A_n)\leq (1-\delta)\log \lambda,\quad \delta=\frac1{300}.\eeqs
Thus we have proved that the LE is discontinuous at $D_s$ in $G^s$ topology.

\section{Proofs of Proposition \ref{lemma.Ilowerbound} and \ref{lemma.Iupperbound}.}
\subsection{Proof of Proposition \ref{lemma.Ilowerbound}.}
We prove the proposition by induction. For sufficiently large $\lambda$, the construction for $n=N$ holds (where we require $\lambda\geq e^{q_N^{q_N}}$), see \cite{WY13} for more details. In the following, we only consider $n\geq N+1$.

Inductively, we assume that $\phi_N$, $\cdots$, $\phi_{n-1}$ have been constructed such that Proposition \ref{lemma.Ilowerbound} holds for $N\leq i\leq n-1$. Moreover, we further assume the following properties.

\noindent
$(1)_i$. For $0<k\leq \log i$,
$$\|\phi_i(x)-\phi_{i-1}\|_{C^k}\leq \lambda_i^{-\frac{r_{i-1}}{100}}.$$
Moreover, for $k\in \mathbb Z_+$,
$$|\partial^k\phi_i(x)|\leq 4^{k\log q_{i+1}}\cdot C^{k^2}\cdot k!\cdot \varphi_{r_i}^k,$$
$$\varphi_{r_i}^k=e^{\sigma k^{\frac{\sigma}{\sigma-1}}}\cdot e^{2k\cdot \big((\log q_i)^{\sigma}\cdot \frac{q_i}{q_{i-1}}+\cdots+(\log q_{N+1})^{\sigma}\cdot \frac{q_{N+1}}{q_{N}}\big)}.$$

\noindent
$(2)_i$. For each $x\in I_i$, $A_i(x),A_i(Tx),\cdots, A_i(T^{r_i^+-1}x)$ is $\lambda_i$-hyperbolic. Moreover, for $x\in I_{i}$ and $k\in \mathbb Z_+$,
$$\Big|\partial^k \|A_i^{r_i^\pm}(x)\|\Big|\leq \|A_i^{r_i^\pm}(x)\|\cdot 4^{k\log q_{i+1}}\cdot C^{k^2}\cdot k!\cdot \varphi_{r_i}^k.$$

\noindent
$(3)_i$. We have
$$s_i(x)-u_i(x)=\phi_0(x),\quad x\in \frac{I_i}{10},$$
$$|s_i(x)-u_i(x)|\geq \frac12|\phi_0(x)|\geq e^{-(-\log (b_i/20))^{\sigma}},\quad x\in I_i\backslash \frac{I_i}{10}.$$

\noindent In this subsection, we will prove the above properties for $i=n$.

\medskip

Firstly, we have a similar lemma as Lemma \ref{lemma.hyperbolic}.
\begin{lemma}\label{lemma.Ihyperbolic}
Let $x_0=x,x_1,\cdots, x_m$ be a $T$-orbit with $x_0,x_m\in I_n$ and $x_i\notin I_n$ for $0<i<m$. Then $\{A_{n-1}(x_0),\cdots A_{n-1}(x_m)\}$ is $\lambda_n$-hyperbolic.
Moreover, letting $\overline s_n (x):= {s(A_{n-1}^{r_n^+}(x))}$ and $\overline u_n (x):= {u(A_{n-1}^{-r_n^-}(x))}$, we have for $x\in I_{n}$ and $1\leq k\leq \log n$,
\beqs
\|\overline s_n-s_{n-1}\|_{C^k},\ \|\overline u_n-u_{n-1}\|_{C^k}\leq \lambda^{-\frac13r_{n-1}}.
\eeqs
Moreover, for $k\in \mathbb Z_+$,
\beqs
\|\overline s_n-s_{n-1}\|_{C^k},\ \|\overline u_n-u_{n-1}\|_{C^k}\leq 8^{k\log r_n^+}\cdot C^{k^2}\cdot k!\cdot \|A_{n-1}^{r_{n-1}^+}(x)\|^{-\frac32}\cdot \varphi_{r_{n-1}}^k,
\eeqs
\beqs
\left|\partial^k\|A_{n-1}^{r_n^+}(x)\|\right|\leq 4^{k\log r_n^+}\cdot C^{k^2}\cdot k!\cdot \|A_{n-1}^{r_n^+}(x)\| \cdot \varphi_{r_n}^k.
\eeqs
\end{lemma}
The proof of Lemma \ref{lemma.Ihyperbolic} is similar to the one of Lemma \ref{lemma.hyperbolic} if we note the angles are not small with respect to the norms
$$e^{-(-\log (b_n/20))^{\sigma}}\geq e^{-(\log(80q_{n+1}))^{\delta^{-\frac12}}}\gg \lambda^{-\eta q_n}=\lambda^{-q_n^{\delta^{\frac12}}},$$
$$e^{-(\log(8 q_{n}))^\sigma}\geq e^{-(\log(8q_{n}))^{\delta^{-\frac12}}}\gg  \lambda^{-\eta q_{n-1}}=\lambda^{-q_{n-1}^{\delta^{\frac12}}},$$
where we have used the Brjuno condition $\log q_{n+1}\leq \beta_\delta q_{n}^\delta$ for large $n$, $\sigma \delta^{\frac12}<1$ and sufficient largeness of $\lambda$.
However, Lemma \ref{lemma.hderivative} is not sufficient to prove Lemma \ref{lemma.Ihyperbolic}. Instead, we need a more delicate analysis on the higher derivatives of norms and angles, which is given by Lemma \ref{lemma.hderivative1}.
\begin{lemma}\label{lemma.hphi0}
Recall $\sigma>1$ and for $x\in [0,2\pi)$,
$$\phi_0(x)=\exp\Big(-\big(\log\frac{8\pi}{x}\big)^{\sigma} -\big(\log\frac{8\pi}{2\pi-x}\big)^{\sigma}\Big).$$
Then there exists a constant $C>0$ such that for $n\in \mathbb N$,
$$\sup_{x\in [0,2\pi)}|\phi_0^{(n)}(x)|\leq e^{Cn^{\frac{\sigma}{\sigma-1}}}.$$
\end{lemma}
\begin{proof}
Equivalently, we consider $$\phi_0(x)=e^{-\big(\log\frac1{x}\big)^{\sigma} },\quad 0<x<\frac14.$$
By Fa\`a di Bruno's formula (Lemma \ref{Faadib})
$$\Big(\big(\log\frac1{x}\big)^\sigma\Big)^{(n)}=\sum_{\substack{k_1+2k_2+\cdots +nk_n=n\\ k=k_1+k_2+\cdots +k_n}} \frac{n!\cdot \sigma (\sigma-1)\cdots(\sigma-k+1)}{k_1!k_2!\cdots k_n!\cdot 1^{k_1}2^{k_2}\cdots n^{k_n}} \big(\log\frac{1}{x}\big)^{\sigma}\cdot \big(\frac{-1}{x}\big)^n.$$
And hence by Lemma \ref{cor.r(1+r)},
\begin{align*}\Big|\Big(\big(\log\frac1{x}\big)^\sigma\Big)^{(n)}\Big|&\leq \sum_{\substack{k_1+2k_2+\cdots +nk_n=n\\ k=k_1+k_2+\cdots +k_n}} \frac{n!\cdot \sigma^n}{k_1!k_2!\cdots k_n!} \big(\log\frac{1}{x}\big)^{\sigma}\cdot \big(\frac{1}{x}\big)^n\\
&\leq n!\cdot (2\sigma)^n\cdot\big(\log\frac{1}{x}\big)^{\sigma}\cdot \big(\frac{1}{x}\big)^n.\end{align*}
Also by Fa\`a di Bruno's formula
\begin{align*}|\phi_0^{(n)}(x)|&\leq \sum_{\substack{k_1+2k_2+\cdots +nk_n=n\\ k=k_1+k_2+\cdots +k_n}} \frac{n!\cdot (2\sigma)^n\cdot \phi_0}{k_1!k_2!\cdots k_n!} \big(\log\frac{1}{x}\big)^{k\sigma}\cdot \big(\frac{1}{x}\big)^n\\
&\leq n!\cdot (2\sigma)^n\cdot\big(2\log\frac{1}{x}\big)^{n\sigma}\cdot \big(\frac{1}{x}\big)^n\cdot e^{-\big(\log\frac1{x}\big)^{\sigma} }.\end{align*}
Let $y=\log\frac1{x}$ and
$$g_n(y):=(2y)^{ n\sigma}e^{ny-y^\sigma}.$$
Since $\sigma>1$, we know $g_n(y)$ reaches its maximum at $y\approx n^{\frac1{\sigma-1}}$ and
$$\sup_{y\geq 1}g_n(y)\leq (2y)^{ n\sigma}e^{(\sigma-1) y^\sigma}\leq e^{(\sigma-1) n^{\frac{\sigma}{\sigma-1}}}.$$
Hence it holds that
$$
\sup_{x}\Big|\phi_0^{(n)}(x)\Big|\leq n^2\cdot \sigma^n\cdot n!\cdot e^{(\sigma-1) n^{\frac{\sigma}{\sigma-1}}}\leq e^{\sigma n^{\frac{\sigma}{\sigma-1}}}.
$$
\end{proof}

\begin{proof3}
By assumption, $x_0,x_m\in I_n$ and $x_i\notin I_n$ for $0<i<m$.
This means $m=r_n^+(x)$. By Lemma \ref{lemma.returntimeIn}, we have $m=q_n$ or $q_{n+1}$.
If $m=q_n$, then $r_n^+(x)=r_{n-1}^+(x)=q_n$. It follows from the assumption that $\{A_{n-1}(x_0),\cdots A_{n-1}(x_m)\}$ is $\lambda_n$-hyperbolic and the required estimates are trivial.

Now assume that $m=q_{n+1}$.
Let $0=j_0<j_1<\cdots<j_p=m$ be the returning times to $I_{n-1}$. Then it holds that $j_t-j_{t-1}\geq q_{n-1}$. Moreover, due to the precise analysis on the returning time in the proof of Lemma \ref{lemma.returntimeIn}, there exists $s\approx p/2$ such that
$$j_t-j_{t-1}= q_n,\quad \text{for}\ t\neq s+1,$$
$$ j_{s+1}-j_s=q_{n-1},\quad |T^{j_s}x|\geq \frac{b_{n-1}}4,\quad |T^{j_{s+1}}x|\geq \frac{b_{n-1}}4.$$
By the inductive assumption, we have for $t\neq s+1$,
$$\|A_{n-1}^{j_t-j_{t-1}}(T^{j_{t-1}}x)\|\geq \lambda_{n-1}^{j_t-j_{t-1}}\geq \lambda_{n-1}^{q_{n}},$$
$$\left|s(A_{n-1}^{j_{t+1}-j_t}(T^{j_{t-1}}x)) -u(A_{n-1}^{j_t-j_{t-1}}(T^{j_{t-1}}x))\right| \geq \frac12 |\phi_0(T^{j_{t-1}}x)|\geq e^{-(-\log (b_n/20))^{\sigma}}.$$
By the definition of Brjuno frequency,
$$e^{-(-\log (b_n/20))^{\sigma}}\geq e^{-(\log(80q_{n+1}))^{\delta^{-\frac12}}}\gg \lambda^{-\eta q_n}=\lambda^{-q_n^{\delta^{\frac12}}}.$$
Hence we have verified the conditions of Lemma \ref{lemma.basic} to the decomposition
$$A_{n-1}^{j_s}(x)=A_{n-1}^{j_s-j_{s-1}}(T^{j_{s-1}}x)\cdots A_{n-1}^{j_2-j_{1}}(T^{j_{1}}x)A_{n-1}^{j_1}(x),$$
$$A_{n-1}^{m-j_{s+1}}(T^{j_{s+1}}x)=A_{n-1}^{j_p-j_{p-1}}(T^{j_{p-1}}x)\cdots A_{n-1}^{j_{s+2}-j_{s+1}}(T^{j_{s+1}}x).$$
As proved in Lemma \ref{lemma.hyperbolic}, we apply Lemma \ref{lemma.basic} in turn to $A_{n-1}^{j_s}(x)$, $A_{n-1}^{m-j_{s+1}}(T^{j_{s+1}}x)$ and $A_{n-1}^m=A_{n-1}^{m-j_{s+1}}(T^{j_{s+1}}x) A_{n-1}^{j_{s+1}-j_{s}}(T^{j_s}x)A_{n-1}^{j_{s}}(x)$.
Then it follows that
\begin{center}
  $\big\{A_{n-1}(x_0),\cdots A_{n-1}(x_m)\big\}$ is $\lambda_n$-hyperbolic,\quad  $\forall x_0,x_m\in I_n$,
\end{center}
$$\|A_{n-1}^{r_n^+}(x)\|\geq \lambda_n^{r_n^+},\quad x\in I_n.$$
On the other hand, it holds by assumption that
$$\Big|\partial^k \|A_{n-1}^{j_t-j_{t-1}}(x)\|\Big|\leq \|A_{n-1}^{j_t-j_{t-1}}(x)\|\cdot 4^{k\log q_{n}}\cdot C^{k^2}\cdot k!\cdot \varphi_{j_t-j_{t-1}}^k,$$
$$\varphi_{j_t-j_{t-1}}^k=e^{\sigma k^{\frac{\sigma}{\sigma-1}}}\cdot e^{2k\cdot \big((\log q_{n-1})^{\sigma}\cdot \frac{q_{n-1}}{q_{n-2}}+\cdots+(\log q_{N+1})^{\sigma}\cdot \frac{q_{N+1}}{q_{N}}\big)},$$
$$|\partial^k\phi_{n-1}(x)|\leq 8^{k\log q_n}\cdot C^{k^2}\cdot k!\cdot\varphi_{r_{n-1}}^k.$$
Then by repeatedly applying Lemma \ref{lemma.hderivative1}, we have
$$\Big| \partial^k\|A_{n-1}^{r_n^+}(x)\|\Big|\leq \|A_{n-1}^{r_n^+}(x)\|\cdot 4^{k\log q_{n+1}}\cdot C^{k^2}\cdot k!\cdot \varphi_{r_n}^k,$$
$$\varphi_{r_n}^k=e^{\sigma k^{\frac{\sigma}{\sigma-1}}}\cdot e^{2k\cdot \big((\log q_{n})^{\sigma}\cdot \frac{q_{n}}{q_{n-1}}+\cdots+(\log q_{N+1})^{\sigma}\cdot \frac{q_{N+1}}{q_{N}}\big)},$$
$$\|\overline s_n-s_{n-1}\|_{C^k},\ \|\overline u_n-u_{n-1}\|_{C^k}\leq 8^{k\log q_{n}}\cdot C^{k^2}\cdot k!\cdot \|A_{n-1}^{r_{n-1}^+}(x)\|^{-\frac32}\cdot \varphi_{r_{n-1}}^k.$$
Specially, if $1\leq k\leq \log n$, then
\begin{align*}
\big|\varphi_{r_n}^k\big|&\leq e^{C_1\log n\cdot (\log \log n)^\frac{\sigma}{\sigma-1}}\cdot e^{2\log n\cdot \big((\log q_{n})^{\sigma}\cdot \frac{q_n}{q_{n-1}}+\cdots+(\log q_{N+1})^{\sigma}\cdot \frac{q_{N+1}}{q_{N}}\big)}\\
&\leq e^{C_1\log n\cdot (\log \log n)^\frac{\sigma}{\sigma-1}}\cdot e^{2q_n\cdot \log n\cdot \big(q_{n-1}^{\delta^{1/2}-1}+q_{n-2}^{\delta^{1/2}-1}\cdot \frac{q_{n-1}}{q_n}+\cdots+q_{N}^{\delta^{1/2}-1}\cdot \frac{q_N}{q_n}\big)}\\
&\leq  e^{\frac18 q_n},
\end{align*}
where we have used Brjuno condition.
This leads to
$$\|\overline s_n-s_{n-1}\|_{C^k},\ \|\overline u_n-u_{n-1}\|_{C^k}\leq \lambda^{-\frac13r_{n-1}}.$$

\end{proof3}

\

Now we are ready for constructing $\phi_n(x)$ and verify $(1)_n$ -- $(3)_n$.

\

\noindent
{\bf \textit{Construction of $\phi_n$ and $A_n$.}}
Applying Lemma \ref{lemma.Ihyperbolic}, we have that for $x\in I_n$,
 \begin{center} $A_{n-1}(x),A_{n-1}(Tx),\cdots, A_{n-1}(T^{r_n^+-1}x)$ is $\lambda_n$-hyperbolic.\end{center}
Let $\overline s_n (x):= {s(A_{n-1}^{r_n^+}(x))}$ and $\overline u_n (x):= {u(A_{n-1}^{-r_n^-}(x))}$.
Let $e_n(x)$ be a $2\pi$-periodic $C^{\infty}$-function such that
$$e_n(x):=\phi_0(x)-(\overline s_n (x)-\overline u_n (x)), \quad x\in I_n.$$
Let $f_n$ be determined by Lemma \ref{lemma.2.3} with $\nu=2$. Define for $x\in \mathbb S^1$,
$$\hat e_n(x)= e_n(x)\cdot f_n(x),\quad \phi_n(x)=\phi_{n-1}(x)+\hat e_n(x),\quad A_n(x)=\Lambda \cdot R_{\frac\pi2-\phi_n(x)}.$$

\noindent
{\bf\textit{Verifying $(1)_n$ of Proposition \ref{lemma.Ilowerbound}.}}
We could apply Lemma \ref{lemma.Ihyperbolic} to obtain
$$\|e_n\|_{C^k}\leq \lambda^{-\frac1{3}q_{n}},\quad \forall x\in \frac{I_n}{10}, \quad 1\leq k\leq \log n.$$
Here we have $r_{n-1}=q_{n}$ since the first returning time of $x\in \frac{I_n}{10}$ to $I_{n-1}$ is $q_n$.
By Lemma \ref{lemma.2.3}, we have
$$\sup_{x\in I_n}|\partial^j f_n|\leq \frac{(Cq_{n+1})^j\cdot (j!)^2}{1+j^2}.$$
Then we have for $1\leq k\leq \log n$ and $x\in \mathbb S^1$,
\beqs
|\partial^k \hat e_n|=\sum_{i+j=k}\frac{k!}{i!\cdot j!}\cdot\big|\partial^ie_n\big|\cdot\big|\partial^jf_n\big|\leq \lambda^{-\frac1{3}q_{n}}\cdot C e^{c \log n\cdot \log q_{n+1}}\leq \lambda^{-\frac1{3}q_{n}}\cdot C e^{c \log n\cdot q_n^{\delta^{\frac12}}}\leq \lambda^{-\frac1{100}q_{n}}.
\eeqs
Hence we have for $1\leq k\leq \log n$,
$$\|\phi_n-\phi_{n-1}\|_{C^k}\leq \lambda_n^{-\frac1{4}q_n}.$$
Moreover, for $k\in \mathbb Z_+$ ($k$ large enough)
\begin{align*}
|\partial^k\phi_n|&\leq |\partial^k\phi_{n-1}|+\sum_{i+j=k}\frac{k!}{i!\cdot j!}\cdot \big|\partial^ie_n\big|\cdot\big|\partial^jf_n\big| \\
&\leq 8^{k\log q_{n}}\cdot C^{k^2}\cdot k!\cdot\varphi_{r_{n-1}}^k+\sum_{i+j=k}k!\cdot 8^{i\log q_{n+1}}\cdot C^{i^2}\cdot \varphi_{r_n}^i\cdot (Cq_{n+1})^j\cdot j!\\
&\leq 8^{k\log q_{n+1}}\cdot  C^{k^2}\cdot k! \cdot \varphi_{r_n}^k.
\end{align*}

\

\noindent
{\bf\textit{Verifying $(2)_n$ of Proposition \ref{lemma.Ilowerbound}.}}
It is clear that $A_n(x)=A_{n-1}(x)R_{-\hat e_n(x)}$. Moreover, it holds
$$A_n^{r_n^+}(x)=A_{n-1}^{r_n^+}(x)\cdot R_{-e_n(x)},\quad A_n^{-r_n^-}(x)=R_{-e_n(T^{-r_n^-}x)}\cdot  A_{n-1}^{-r_n^-}(x).$$
Since the rotations do not affect the norms, by Lemma \ref{lemma.Ihyperbolic} for each $x\in I_n$,

\begin{center} $A_n(x),A_n(Tx),\cdots, A_n(T^{r_n^+-1}x)$ is $\lambda_n$-hyperbolic,\end{center}
\beqs
\left|\partial^k\|A_{n}^{r_n^+}(x)\|\right|\leq 4^{k\log r_n}\cdot C^{k^2}\cdot k!\cdot \|A_{n-1}^{r_n^+}(x)\| \cdot \varphi_{r_n}^k.
\eeqs

\

\noindent
{\bf\textit{Verifying $(3)_n$ of Proposition \ref{lemma.Ilowerbound}.}}
It is clear that
$$s_n(x)-u_n(x)=\phi_0(x),\quad x\in \frac{I_n}{10}.$$
Moreover, it holds that
$$|s_n(x)-u_n(x)|\geq |\phi_0(x)|-\lambda_{n}^{-q_n}\geq C e^{-(-\log (b_n/20))^{\sigma}},\quad x\in I_n\backslash \frac{I_n}{10}.$$

\

\subsection{Proof of Proposition \ref{lemma.Iupperbound}.}
Let $\widetilde e_n(x)=-(s_n(x)-u_n(x))\cdot f_n(x)$ be a $2\pi$-periodic smooth function such that it is $-(s_n(x)-u_n(x))$
on $\frac{I_n}{10}$ and vanishes outside $I_n$. Hence by Proposition \ref{lemma.Ilowerbound}, $\widetilde e_n(x)=-\phi_0(x)\cdot f_n(x)$ on $\frac{I_n}{10}$.
Define
$$\widetilde \phi_n(x)=\phi_n(x)+\widetilde e_n(x),\quad \widetilde A_n(x)=\Lambda \cdot R_{\frac\pi2-\widetilde \phi_n(x)}.$$

\

\noindent
{\bf\textit{Verifying $(1)_n$ of Proposition \ref{lemma.Iupperbound}.}}
By Lemma \ref{lemma.2.3} with parameter $\nu=2$, we have
$$\sup_{x\in I_n}|\partial^j f_n|\leq \frac{(Cq_{n+1})^j\cdot (j!)^2}{1+j^2}.$$
From the proof of Lemma \ref{lemma.hphi0},
\begin{align*}|\partial^k\phi_0(x)| \leq k!\cdot (2\sigma)^k\cdot\big(2\log\frac{1}{x}\big)^{k\sigma}\cdot \big(\frac{1}{x}\big)^k\cdot e^{-\big(\log\frac1{x}\big)^{\sigma} }.\end{align*}
And the function on the righthand side is decreasing on $I_n$.
Then it implies for $0\leq k\leq \log n$ and $x\in I_n$,
$$|\partial^k \phi_0(x)|\leq [\log n]!\cdot (2\sigma)^{\log n}\cdot (4\log q_{n+1})^{\sigma \log n}\cdot q_{n+1}^{\log n}\cdot e^{-(\log q_{n+1})^\sigma}\leq e^{-\frac23(\log q_{n+1})^\sigma}.$$
Hence for $0\leq k\leq \log n$,
\begin{align*}
|\partial^k(\widetilde \phi_n-\phi_n)|&=|\partial^k\widetilde e_n|\leq \sum_{i+j=k}\frac{k!}{i!\cdot j!}\cdot \big|\partial^i\phi_0\big|\cdot\big|\partial^jf_n\big| \\
&\leq 2^{\log n}\cdot e^{-\frac23(\log q_{n+1})^\sigma}\cdot (\log n)^{2\log n}\cdot q_{n+1}^{\log n} \leq e^{-\frac12(\log q_{n+1})^\sigma}\leq q_{n+1}^{-2}.
\end{align*}

\

\noindent
{\bf\textit{Verifying $(2)_n$ of Proposition \ref{lemma.Iupperbound}.}}
Since $\widetilde \phi_n(x)=\phi_n(x)$ on $\mathbb S^1\backslash I_n$ and $(2)_n$ in Proposition \ref{lemma.Ilowerbound}, we have
that the sequence $\{\widetilde A_n(x), \widetilde A_n(Tx),\cdots,\widetilde A_n(T^{r_n^+-1}x)\}$ is $\lambda_n$-hyperbolic for $x\in I_n$.

\

\noindent
{\bf\textit{Verifying $(3)_n$ of Proposition \ref{lemma.Iupperbound}.}}
Note that $\widetilde s_n(x)-\widetilde u_n(x)=s_n(x)-u_n(x)-\widetilde e_n(x)$. Hence
$$\widetilde s_n(x)=\widetilde u_n(x),\quad x\in \frac{I_n}{10}.$$

\section{Proofs of Proposition \ref{lemma.Glowerbound} and \ref{lemma.Gupperbound}.}
\subsection{Preparation lemmas.}
Recall that for any $f\in G^s(I)$,
$$\|f\|_{s,K}:=\frac{4\pi^2}{3} \sup_k \frac{(1+|k|^2)}{K^k(k!)^s}\|\partial^kf\|_{C^0}.$$
We have the following lemmas, which is given in \cite{GWYZ}.
\begin{lemma}[Proposition 2.1 in \cite{GWYZ}]\label{lemma.2.1}
Assume $f,g\in G^{s,K}(I)$ and $\varepsilon>0$ is sufficiently small, we have
\begin{description}
  \item[(1)] $\|fg\|_{s,K}\leq \|f\|_{s,K}\|g\|_{s,K}.$
  \item[(2)] $\|\partial f\|_{s,(1+\varepsilon^{1/2})K}\leq K\varepsilon^{-1}.$
  \item[(3)] Assume $\|f-1\|_{s,K}\leq \varepsilon$, then $$\|f^{-1}-1\|_{s,(1+\varepsilon^{(s+8)^{-1}})K}\leq \varepsilon^{1/12}.$$
  \item[(4)] Assume $\|f-1\|_{s,K}\leq \varepsilon$, then $$\|f^{1/2}-1\|_{s,(1+\varepsilon^{(s+16)^{-1}})K}\leq \varepsilon^{1/12}.$$
  \item[(5)] Assume $\|f-1\|_{s,K}\leq \varepsilon$, then $\arcsin(f)\in G^{s,4K}(I)$ and $$\|\sin f\|_{s,(1+\varepsilon^{(s+8)^{-1}})K},\ \|\cos f-1\|_{s,(1+\varepsilon^{(s+8)^{-1}})K}\leq \varepsilon^{1/12}.$$
\end{description}
\end{lemma}

\begin{lemma}[Corollary 2.2 in \cite{GWYZ}]\label{lemma.2.2}
Assume that $1<\nu<\infty$ and
$$g(x)=\left\{
         \begin{array}{ll}
           c\exp\left(-x^{\frac{-1}{\nu-1}} -(2\pi-x)^{\frac{-1}{\nu-1}}\right), & \hbox{$x\in(0, 2\pi)$;} \\
           0, & \hbox{$x=0$.}
         \end{array}
       \right.
$$
then there is some $C>0$ such that for any $x\in \mathbb S^1$, we have  for all $n\in \mathbb N$,
$$|g^{(n)}(x)|\leq C^n \exp\left(-|x|^{\frac{-1}{\nu-1}} -|2\pi-x|^{\frac{-1}{\nu-1}}\right)(n!)^{\nu}.$$
As a corollary, $g\in G^{\nu}(\mathbb S^1)$.
\end{lemma}

\begin{lemma}[A revised version of Lemma 2.3 in \cite{GWYZ}]\label{lemma.2.3}
Assume that $1<\nu <s<\infty$. For any $n\geq N$, there exist an absolute constant $C$ and a $2\pi$-periodic function $f_n\in G^s(\mathbb S^1)$ such that
$$
f_n(x)\left\{
         \begin{array}{ll}
           =1, & \hbox{$x\in \frac{I_n}{10}$;} \\
           \in (0,1], & \hbox{$x\in I_n\backslash \frac{I_n}{10}$;} \\
           =0, & \hbox{$x\in \mathbb S^1\backslash I_n$,}
         \end{array}
       \right.
$$
and
$$\|f_n\|_{s,C}\leq e^{(s-v)b_n^{\frac{-1}{s-\nu}}}.$$
\end{lemma}
\begin{proof}
Define
$$\phi(x)=\left\{
            \begin{array}{ll}
              \exp \left(-x^{-\frac1{\nu-1}}\right), & \hbox{$x>0$;} \\
              0, & \hbox{$x\leq 0$,}
            \end{array}
          \right.
$$
$$w_0(x)=\frac{\phi(x+2)}{\phi(x+2)+\phi(-x-1)}, \quad w_1(x)=\left\{
            \begin{array}{ll}
              w_0(-x), & \hbox{$x>0$;} \\
              w_0(x), & \hbox{$x\leq 0$.}
            \end{array}
          \right.$$
By Lemma \ref{lemma.2.3}, it is straightforward to compute $w_1\in G^{\nu}(\mathbb R)$. Then we define $f_n$ to be a $2\pi$-periodic functions such that
$$f_n(x)=w_1(10b_n^{-1}x),\quad x\in [-\pi,\pi].$$
Since $w_1\in G^{\nu}(\mathbb R)$, it holds that
$$\sup_{x\in I_n} |f_n^{(r)}(x)|\leq \frac{(Cb_n^{-1})^r(r!)^{\nu}}{1+r^2}.$$
Note that the function $f(r):=a^rr^{-r}$ reaches its maximum at $r=ae^{-1}$. Thus
$$\sup_{x\in I_n} \frac{|f_n^{(r)}(x)|(1+r^2)}{C^r(r!)^{s}}\leq b_n^{-r}(r!)^{\nu-s}\leq \left(e^{-1}rb_n^{\frac{1}{s-\nu}}\right)^{-r(s-\nu)}\leq e^{(s-v)b_n^{\frac{-1}{s-\nu}}}.$$
\end{proof}

\begin{lemma}[A revised version of Lemma 4.1 in \cite{GWYZ}]\label{lemma.4.1}
Suppose $s>2$ and $\lambda$ is sufficiently large.
Let
$$E(x)=\left(
      \begin{array}{cc}
        e_2(x) & 0 \\
        0 & e_2(x)^{-1} \\
      \end{array}
    \right)R_{\frac\pi2-\theta(x)}\left(
      \begin{array}{cc}
        e_1(x) & 0 \\
        0 & e_1(x)^{-1} \\
      \end{array}
    \right),
$$
$$ e(x):=\|E(x)\|,\quad s(x):=s(E(x)),\quad u(x):=u(E(x)),$$
where $e_1,e_2,\theta\in G^{s,K}(I)$ satisfying
$$\inf_{x\in I}|\theta(x)|\geq ce^{-b_n^{-\sigma}}\gg \min \left\{\inf_{x\in I}e_1(x),\inf_{x\in I}e_2(x)\right\}^{-\frac1{100}},\quad \sigma:=\frac1{s-1},$$
$$\left\|\frac1{\sin\theta}\right\|_{s,K},\ \|\cot\theta\|_{s,K}\leq C e^{b_n^{-\sigma}},\ \|\sin\theta\|_{s,K}, \ \|\tan\theta\|_{s,K}\leq C,$$
$$\|e_j^{-1}\|_{s,K}\leq C\lambda^{-\frac13 q_{n-1}}, \ \ j=1,2.$$
Then it holds that
$$e(x)\geq c\cdot e_1(x)\cdot e_2(x)\cdot \theta(x),$$
$$\|e\|_{s,(1+\eta)K}\leq C^2 \|e_1\|_{s,K}\cdot\|e_2\|_{s,K},$$
$$\|e^{-1}\|_{s,(1+\eta)K}\leq C^2\|e_1^{-1}\|_{s,K}\cdot\|e_2^{-1}\|_{s,K}\cdot e^{b_n^{-\sigma}},$$
$$\left\|\pi/2-s\right\|_{s,(1+\eta)K}\leq \|e_1^{-1}\|_{s,K}^3\cdot \|e_1\|_{s,K}^2,\quad \left\|u\right\|_{s,(1+\eta)K}\leq \|e_2^{-1}\|_{s,K}^3\cdot \|e_2\|_{s,K}^2,$$
with $\eta=\lambda^{-\frac1{4000}q_{n-1}}$.
\end{lemma}

\begin{lemma}[A revised version of Lemma 4.2 in \cite{GWYZ}]\label{lemma.4.2}
Consider a sequence of maps
$$A^l\in G^{s,K}(I,\mathrm{SL}(2,\mathbb R)),\quad 0\leq l <m-1\leq q_n^C.$$
and their products
$$A^m(x)=A_{m-1}(x)A_{m-2}(x)\cdots A_1(x)A_0(x),$$$$ A^{-m}(x)=A_{-m}^{-1}(x)A_{-m+1}^{-1}(x)\cdots A_{-2}^{-1}(x) A_{-1}^{-1}(x).$$
Denote $a(x)=\|A^m(x)\|$, $s(x)=s(A^m(x))$, $u(x)=u(A^m(x))$ and $a_l(x)=\|A_l(x)\|$.
Let $1<\xi<\frac1{100}$ and $\theta_l(x):=u(A_{l-1}(x))-s(A_l(x))$. Assume that
$$\inf_{x\in I} |\theta_l(x)|\geq ce^{-b_n^{-\sigma}}\gg \min_{0\leq l\leq m-1}\bigg(\inf_{x\in I} a_l(x)\bigg)^{-\frac1{100}},$$
$$\left\|\csc\theta_l\right\|_{s,K},\ \|\cot\theta_l\|_{s,K}\leq C e^{b_n^{-\sigma}},\ \|\sin\theta_l\|_{s,K}, \ \|\tan\theta_l\|_{s,K}\leq C,$$
$$\|a_l^{-1}\|_{s,K}\leq \bigg(\|a_l\|_{s,K}\bigg)^{-1+\xi} \leq C\lambda^{-\frac13 q_{n}}.$$
Then it holds that
$$\inf_{x\in I} a(x)\geq c^m e^{-2mb_n^{-1}}\prod_{l=0}^{m-1}\inf_{x\in I} a_l(x),$$
$$\|a\|_{s,(1+\eta)K}\leq C^{2m} \prod_{l=0}^{m-1}\|a_l\|_{s,K},$$
$$\|a^{-1}\|_{s,(1+\eta)K}\leq C^{2m} e^{4mb_n^{-\sigma}}\prod_{l=0}^{m-1} \|a_l^{-1}\|_{s,K},$$
$$\|s(\cdot)-s(A_0(\cdot))\|_{s,(1+\eta)K}\leq \lambda^{-\frac1{10}q_{n}},\quad \|u(\cdot)-u(A_{m-1}(\cdot))\|_{s,(1+\eta)K}\leq \lambda^{-\frac1{10}q_{n}},$$
with $\eta=\lambda^{-\frac1{4000}q_{n}}$.
\end{lemma}
\begin{remark}
The key estimate to prove of Lemma \ref{lemma.4.1} and Lemma \ref{lemma.4.2} is the following. Recall $\alpha\in DC_{\gamma,\tau}$ and hence $q_{n+1}\leq \gamma^{-1} q_n^\tau$. Provided $s>\tau+1$, $\nu\in [1,s-\tau)$ and $\lambda$ sufficiently large, it holds that
$$\lambda^{-c_1q_n}\cdot e^{c_2b_n^{\frac{-1}{s-\nu}}}\leq e^{4c_2q_{n+1}^{\frac{1}{s-\nu}}-c_1\log\lambda \cdot q_n}\leq e^{4c_2\gamma^{\frac{-1}{s-\nu}}q_{n}^{\frac{\tau}{s-\nu}}-c_1\log\lambda \cdot q_n}\leq e^{-c_3\log \lambda q_n}=\lambda^{-c_3 q_n},$$
where $c_1$, $c_2$ and $c_3$ are positive constants.
Since the proofs are nearly the same as the one in \cite{GWYZ} whenever one use the above estimate, we omit them here.
\end{remark}

Now we assume that $1<\nu<s_1<s_2$ and $s_2-\nu>s_1-1>0$. We are going to prove Proposition \ref{lemma.Glowerbound} with $s=s_1$ and Proposition \ref{lemma.Gupperbound} with $s=s_2$. Since $\|f\|_{s_2,K}\leq \|f\|_{s_1,K}$ for any $f\in G^{s_1,K}$, we finish the proof by choosing $s=s_2$.

\subsection{Proof of Proposition \ref{lemma.Glowerbound}.}

We prove the proposition by induction. For sufficiently large $\lambda$, the construction for $n=N$ holds, see \cite{GWYZ} for more details
(we require $\lambda>e^{q_N^{q_N}}$ in this case).
In the following, we only focus on $n\geq N+1$.

Inductively, we assume that $\phi_N$, $\cdots$, $\phi_{n-1}$ have been constructed such that Proposition \ref{lemma.Glowerbound} holds for $N\leq i\leq n-1$, i.e.

\noindent
$(1)_i$. Set $\eta_i=\prod_{j=N}^{i-1}(1+\lambda^{-\frac1{8000}q_{j}})-1$. Then
$$\|\phi_i(x)-\phi_{i-1}\|_{s_1,(1+\eta_i)C}\leq \lambda_i^{-\frac{q_{i-1}}{10}}.$$

\noindent
$(2)_i$. For each $x\in I_i$, $A_i(x),A_i(Tx),\cdots, A_i(T^{r_i^+(x)-1}x)$ is $\lambda_i$-hyperbolic. Moreover, for $x\in I_{i}$,
$$\Big\| \|A_i^{r_i^+ }(x)\|\Big\|_{G^{s_1,(1+\eta_i)K}(I_i)}\leq \widetilde\lambda_i^{r_i^\pm},\quad
\Big\|\|A_i^{r_i^+ }(x)\|^{-1}\Big\|_{G^{s_1,(1+\eta_i)K}(I_i)}\leq \lambda_i^{-r_i^\pm}.
$$

\noindent
$(3)_i$. We have
$$s_i(x)-u_i(x)=\phi_0(x),\quad x\in \frac{I_i}{10},$$
$$|s_i(x)-u_i(x)|\geq \frac12\phi_0(x)\geq \frac12e^{-cb_i^{\frac{-1}{s_1-1}}},\quad x\in I_i\backslash \frac{I_i}{10}.$$
Now we construct $\phi_n(x)$ and verify $(1)_n$ -- $(3)_n$.

For any $x\in I_n$, Let $0=j_0<j_1<\cdots<j_p=r_n^{+}(x)$ be the returning times to $I_{n-1}$.
By Lemma \ref{lemma.returntimeIn}, we have $r_n^{+}(x)=q_{n+1}$ or $q_n$.
If $r_n^{+}(x)=q_n$, then $(1)_n$ -- $(3)_n$ hold by the previous step.
If $r_n^{+}(x)=q_{n+1}$, then also by Lemma \ref{lemma.returntimeIn}
\beq\label{p}p\leq q_n^{-1}q_{n+1}.\eeq

\

\noindent
\textit{Construction of $\phi_n$ and $A_n$.}
By $(2)_{n-1}$, it holds that
$$\big\|A_{n-1}^{r_{n-1}^+ }(x)\big\|\cdot e^{-cb_{n}^{\frac{-1}{s_1-1}}}\geq \lambda_{n-1}^{r_{n-1}^+ }\cdot e^{-cb_{n}^{\frac{-1}{s_1-1}}}\geq \lambda_{n-1}^{(1-\varepsilon)r_{n-1}^+ },\quad x\in I_{n-1}.$$
Combined this with $(3)_{n-1}$ to apply Lemma \ref{lemma.4.2}, we obtain that for each $x\in I_n$,
 \begin{center} $A_{n-1}(x),A_{n-1}(Tx),\cdots, A_{n-1}(T^{r_n^+-1}x)$ is $\lambda_n$-hyperbolic.\end{center}
Let $\overline s_n (x):= {s(A_{n-1}^{r_n^+}(x))}$ and $\overline u_n (x):= {u(A_{n-1}^{-r_n^-}(x))}$.
Let $e_n(x)$ be a $2\pi$-periodic $C^\infty$-function such that
$$e_n(x):=\phi_0(x)-(\overline s_n (x)-\overline u_n (x)), \quad x\in I_n.$$
Define for $x\in \mathbb S^1$,
$$\hat e_n(x)= e_n(x)\cdot f_n(x),\quad \phi_n(x)=\phi_{n-1}(x)+\hat e_n(x),\quad A_n(x)=\Lambda \cdot R_{\frac\pi2-\phi_n(x)}.$$

\

\noindent
\textit{Verifying $(1)_n$ of Proposition \ref{lemma.Glowerbound}.}
Let $\eta_n=\prod_{j=N}^{n-1}(1+\lambda^{-\frac1{8000}q_{j}})-1$, $a_l(x)=\|A_{n-1}^{j_{l+1}-j_l}(T^{j_l}x)\|$ and $I=I_n$.
Now we are going to verify the conditions of Lemma \ref{lemma.4.2}.
By $(2)_{n-1}$, we have
$$\inf_{x\in I}\|A_{n-1}^{j_{l+1}-j_l}(T^{j_l}x)\|\geq \lambda_{n-1}^{j_{l+1}-j_l},\quad 0\leq l<p.$$
By $(3)_{n-1}$, we have
\begin{align*}
&|\theta_l(x)|=|s(A_{n-1}^{j_{l+1}-j_l}(T^{j_l}x))-u(A_{n-1}^{j_l-j_{l-1}}(T^{j_{l-1}}x))|\\
&\geq\left\{
        \begin{array}{ll}
          |s_{n-1}(T^{j_l}x)-u_{n-1}(T^{j_l}x)|-C\lambda_{n-1}^{-q_{n}}\geq e^{-b_{n}^{\frac{-1}{s_1-1}}}, & \hbox{if $j_{l+1}-j_l,j_l-j_{l-1}=q_n$;} \\
          |s_{n-1}(T^{j_l}x)-u_{n-1}(T^{j_l}x)|-C\lambda_{n-1}^{-q_{n-1}}\geq e^{-b_{n-1}^{\frac{-1}{s_1-1}}}, & \hbox{if $j_{l+1}-j_l=q_{n-1}$ or $j_l-j_{l-1}=q_{n-1}$.}
        \end{array}
      \right.
\end{align*}
Also by $(2)_{n-1}$, we have
$$\|a_l\|_{G^{s_1,(1+\eta_{n-1})C}} \cdot \|a_l^{-1}\|_{G^{s_1,(1+\eta_{n-1})C}}\leq \left(\frac{\widetilde\lambda_{n-1}}{\lambda_{n-1}}\right)^{j_{l+1}-j_l}\leq \lambda^{4\varepsilon (j_{l+1}-j_l)}\leq \|a_l\|_{G^{s_1,(1+\eta_{n-1})C}}^\xi,$$
$$\|a_l\|_{G^{s_1,(1+\eta_{n-1})C}}^{-1+\xi}\leq \lambda_{n-1}^{-\frac 12 q_{n}}.$$
By $(1)_{n-1}$, we have $$\|\phi_{n-1}-\phi_0\|_{G^{s_1,(1+\eta_{n-1})C}}\leq 2\lambda^{-\frac1{100}}.$$
By Lemma \ref{lemma.2.1}, we also have
$$\|\sin\theta_l\|_{G^{s_1,(1+\eta_{n-1})C}},\quad \|\tan\theta_l\|_{G^{s_1,(1+\eta_{n-1})C}}\leq C,$$
$$\|\csc\theta_l\|_{G^{s_1,(1+\eta_{n-1})C}},\quad \|\cot\theta_l\|_{G^{s_1,(1+\eta_{n-1})C}}\leq Ce^{b_n^{-\frac1{s_1-1}}}.$$
Then we could apply Lemma \ref{lemma.4.2} and (\ref{p}) to obtain
$$\left\|\|A_{n-1}^{r_n^+}\|\right\|_{G^{s_1,(1+\eta_n)C}}\leq C^{2p}\widetilde \lambda_{n-1}^{r_n^+}\leq \widetilde \lambda_n^{r_n^+},$$
$$\left\|\|A_{n-1}^{r_n^+}\|^{-1}\right\|_{G^{s_1,(1+\eta_n)C}}\leq C^{2p}e^{4pb_n^{-\frac1{s_1-1}}} \lambda_{n-1}^{-r_n^+}\leq \lambda_n^{-r_n^+},$$
$$\|e_n\|_{G^{s_1,(1+\eta_n)C}}\leq \lambda_n^{-\frac1{20}q_{n}}.$$
Moreover, by the definition of $\phi_n$ and Lemma \ref{lemma.2.3}, we have
$$\|\phi_n-\phi_{n-1}\|_{G^{s_1,(1+\eta_n)C}}\leq \|e_n\|_{G^{s_1,(1+\eta_n)C}}\cdot \|f_n\|_{G^{s_1,(1+\eta_n)C}}\leq \lambda_n^{-\frac1{20}q_{n}}e^{(s_1-\nu)b_n^{\frac{-1}{s_1-\nu}}}\leq \lambda_n^{-\frac1{40}q_{n}}.$$

\

\noindent
\textit{Verifying $(2)_n$ of Proposition \ref{lemma.Glowerbound}.}
It is clear that $A_n(x)=A_{n-1}(x)R_{-\hat e_n(x)}$. Moreover, it holds
$$A_n^{r_n^+}(x)=A_{n-1}^{r_n^+}(x)\cdot R_{-e_n(x)},\quad A_n^{-r_n^-}(x)=R_{-e_n(T^{-r_n^-}x)}\cdot  A_{n-1}^{-r_n^-}(x).$$
Since the rotations do not affect the norms, for each $x\in I_n$, $A_n(x),A_n(Tx),\cdots, A_n(T^{r_n^+-1}x)$ is $\lambda_n$-hyperbolic.

\

\noindent
\textit{Verifying $(3)_n$ of Proposition \ref{lemma.Glowerbound}.}
It is clear that
$$s_n(x)-u_n(x)=\phi_0(x),\quad x\in \frac{I_n}{10}.$$
Moreover, it holds that
$$|s_n(x)-u_n(x)|\geq |\phi_0(x)|-\lambda_{n}^{-q_n}\gtrsim e^{-b_{n}^{\frac{-1}{s_1-1}}},\quad x\in I_n\backslash \frac{I_n}{10}.$$

\

\subsection{Proof of Proposition \ref{lemma.Gupperbound}.}
Let $\widetilde e_n(x)=-(s_n(x)-u_n(x))\cdot f_n(x)$ be a $2\pi$-smooth function. Hence by Proposition \ref{lemma.Glowerbound}, $\widetilde e_n(x)=-\phi_0(x)\cdot f_n(x)$.
Define
$$\widetilde \phi_n(x)=\phi_n(x)+\widetilde e_n(x),\quad \widetilde A_n(x)=\Lambda \cdot R_{\frac\pi2-\widetilde \phi_n(x)}.$$

\

\noindent
\textit{Verifying $(1)_n$ of Proposition \ref{lemma.Gupperbound}.}
By Lemma \ref{lemma.2.2} and Lemma \ref{lemma.2.3}, we have
$$\|\widetilde \phi_n-\phi_n\|_{G^{s_2,C}}=\|\widetilde e_n\|_{G^{s_2,C}}\leq \|\phi_0\|_{G^{s_1,C}}\cdot \|f_n\|_{G^{s_2,C}} \leq e^{-b_n^{\frac{-1}{s_1-1}}}\cdot e^{(s_2-\nu)b_n^{\frac{-1}{s_2-\nu}}}\leq e^{-\frac12b_n^{\frac{-1}{s_1-1}}}\leq q_{n+1}^{-2}.$$
Here we choose $\nu>1$ sufficiently close to $s_2$ such that $s_2-\nu>s_1-1>0$.

\

\noindent
\textit{Verifying $(2)_n$ of Proposition \ref{lemma.Gupperbound}.}
Since $\widetilde \phi_n(x)=\phi_n(x)$ on $\mathbb S^1\backslash I_n$ and $(2)_n$ in Proposition \ref{lemma.Glowerbound}, we have
that the sequence $\{\widetilde A_n(x), \widetilde A_n(Tx),\cdots,\widetilde A_n(T^{r_n^+-1}x)\}$ is $\lambda_n$-hyperbolic for $x\in I_n$.

\

\noindent
\textit{Verifying $(3)_n$ of Proposition \ref{lemma.Gupperbound}.}
Note that $\widetilde s_n(x)-\widetilde u_n(x)=s_n(x)-u_n(x)-\widetilde e_n(x)$. Hence
$$\widetilde s_n(x)=\widetilde u_n(x),\quad x\in \frac{I_n}{10}.$$

\section*{Appendix A. Higher derivatives for the angles and the norms.}
\setcounter{equation}{0}
\renewcommand{\theequation}{A.\arabic{equation}}
\setcounter{theorem}{0}
\renewcommand{\thesection}{A}

\begin{proof1}
Let
$$D:=\left(
      \begin{array}{cc}
        e_2 & 0 \\
        0 & e_2^{-1} \\
      \end{array}
    \right)R_{\frac\pi2-\theta}\left(
      \begin{array}{cc}
        e_1 & 0 \\
        0 & e_1^{-1} \\
      \end{array}
    \right),
\quad s(x):=s(D(x)),\quad u(x):=u(D(x)).$$
Then $$s_3-s_1=s-\frac\pi2,\quad u_3-u_2=u,\quad e_3=\|D\|.$$
A direct computation shows that
\beqs
D^tD=\left(
      \begin{array}{cc}
        e_1^2e_2^2\sin^2\theta+e_1^2e_2^{-2}\cos^2\theta & (e_2^{-2}-e_2^2)\cos\theta\sin\theta \\
        (e_2^{-2}-e_2^2)\cos\theta\sin\theta & e_1^{-2}e_2^{-2}\sin^2\theta+e_2^2e_1^{-2}\cos^2\theta \\
      \end{array}
    \right).
\eeqs
First note
\beqs
e_3^2(x)+e_3^{-2}(x)=e_1^2e_2^2\sin^2\theta+e_1^2e_2^{-2}\cos^2\theta +e_1^{-2}e_2^2\cos^2\theta+e_1^{-2}e_2^{-2}\sin^2\theta.
\eeqs
If $$e_0\gg 1,\quad |\theta|\geq e_0^{-\eta},$$
then we have $e_3^2+e_3^{-2}\geq e_1^2e_2^2\sin^2\theta \gg 1$ and hence $e_3 \gg 1$.
Then
$$e_3= e_1e_2|\sin\theta|+O(e_0^{-\frac12}).$$

Next we compute $s(x)$.
Let $a=e_1e_2$ and $b=e_1e_2^{-1}$.
It is straightforward to compute
\beq\label{sformula}\tan s(x)=\sqrt{\frac{U^2}{w^2}+1}+\frac Uw=\Bigg(\sqrt{\frac{U^2}{w^2}+1}-\frac Uw\Bigg)^{-1},\eeq
where
$$w=2(e_2^2-e_2^{-2})\cos\theta\sin\theta,\quad U=(a^2-a^{-2})\sin^2\theta+(b^2-b^{-2})\cos^2\theta.$$
Hence
$$\frac\pi2-s(x)=\arctan \Bigg(\sqrt{\frac{U^2}{w^2}+1}-\frac Uw\Bigg).$$
Then we have
\begin{align*}
\frac{U}{w}&=\frac{e_1^2e_2^2-e_1^{-2}e_2^{-2}}{2(e_2^2-e_2^{-2})}\tan\theta +\frac{e_1^2e_2^{-2}-e_1^{-2}e_2^2}{2(e_2^2-e_2^{-2})}\cot\theta\\
&=\frac12e_1^2\cdot (1-e_2^{-4})^{-1} \cdot\left((1-e_1^{-4}e_2^{-4})\tan\theta +(e_2^{-4}-e_1^{-4})\cot\theta\right).
\end{align*}
Note $|\theta|\geq e_0^{-\eta}$. Then it is easy to compute that
\begin{align*}
\frac{U}{w}\geq \frac12 e_1^{2-\eta}.
\end{align*}
Hence we have
$$\left|\frac\pi2-s(x)\right|\leq e_1^{-2+\eta}.$$
This implies that
$$\left|s_3(x)-s_1(x)\right|\leq e_1^{-2+\eta}.$$
Similarly, we could also obtain
$$\left|u_3(x)-u_2(x)\right|\leq e_2^{-2+\eta}.$$
\end{proof1}

\begin{lemma}\label{lemma.hderivative1}
Let $l$ be a finite positive integer or infinity, $C$ be a universal constant and
$$e_j(x)=\|E_j(x)\|,\quad s_j(x)=s(E_j(x)),\quad u_j(x)=u(E_j(x)), \quad j=1,2,3,$$
$$\theta(x)=s_2(x)-u_1(x),\quad e_0=\min\{e_1,e_2\}.$$
Assume that for $x\in I$,
$$|\partial_x^ke_j|\leq C^{k^2}\cdot k!\cdot  e_j\cdot \varphi_j^k,\quad |\partial_x^k\theta|\leq C^{k^2}\cdot k!\cdot \mu_k^k,\quad  j=1,2,\quad k=1,2,3,\cdots,l,$$
$$\inf_{x\in I}e_0(x)\gg 1,\quad  |\theta|\geq e_0^{-\frac1{100}}.$$
Let $$ \varphi_3:=\max\{\varphi_1,\varphi_2,\max_{1\leq j\leq k}\mu_j\cdot |\theta |^{-2}\},\quad \tilde \varphi_3:=\max\{\varphi_1,\varphi_2,\max_{1\leq j\leq k}\mu_j\}.$$
Then we have for $k=1,2,3,\cdots,l$ and $x\in I$,
$$ |\partial_x^ke_3|\leq 4^k\cdot C^{k^2}\cdot k!\cdot e_3\cdot \varphi_3^k,$$
$$\|u_3-u_2\|_{C^k}\leq 8^k\cdot C^{k^2}\cdot k!\cdot e_2^{-\frac32}\cdot \tilde \varphi_3^k,\quad \|s_3-s_1\|_{C^k}\leq 8^k\cdot C^{k^2}\cdot k!\cdot e_1^{-\frac32}\cdot \tilde\varphi_3^k.$$
\end{lemma}
\begin{proof}
We first state some lemmas for higher derivatives.
\begin{lemma}[The formula of Fa\`a di Bruno]\label{Faadib}
Assume $f$ and $g$ are two smooth functions in an open interval $(a,b)$, let $h=g\circ f$, then
$$
h^{(n)}(x)=\sum_{\substack{k_1+2k_2+\cdots +nk_n=n\\ k=k_1+k_2+\cdots +k_n}} \frac{n!}{k_1!k_2!\cdots k_n!}g^{(k)}\big(f(x)\big)\Bigg(\frac{f^{(1)}}{1!}\Bigg)^{k_1} \Bigg(\frac{f^{(2)}}{2!}\Bigg)^{k_2}\cdots \Bigg(\frac{f^{(n)}}{n!}\Bigg)^{k_n}.
$$
\end{lemma}

\begin{lemma}\label{cor.r(1+r)}
$$
\sum_{\substack{k_1+2k_2+\cdots +nk_n=n\\ k=k_1+k_2+\cdots +k_n}}\frac{k!}{k_1!k_2!\cdots k_n!}R^k=R(1+R)^{n-1}.
$$
\end{lemma}

Now we start the proof.
Note that
\beq\label{compe3}
e_3^2+e_3^{-2}=e_1^2e_2^2\sin^2\theta+e_1^2e_2^{-2}\cos^2\theta +e_1^{-2}e_2^2\cos^2\theta+e_1^{-2}e_2^{-2}\sin^2\theta.
\eeq
Take $n$th derivative on both sides. Then by Fa\`a di Bruno's formula, we have for $m=1,2,3$,
$$\partial^n\big(e_m^2\big)=\sum_{j=1}^{n-1} n!\cdot\frac{e_m^{(j)}}{j!}\cdot \frac{e_m^{(n-j)}}{(n-j)!}+2e_m\cdot e_m^{(n)},$$
\begin{align*} \partial^n(e_m^{-2})= \sum_{\substack{k_1+2k_2+\cdots +nk_n=n\\ 2\leq k=k_1+k_2+\cdots +k_n}} \frac{n!(-1)^k (k+1)!}{k_1!k_2!\cdots k_n!e_m^{k+2}}\bigg(\frac{e_m^{(1)}}{1!}\bigg)^{k_1} \bigg(\frac{e_m^{(2)}}{2!}\bigg)^{k_1}\cdots \bigg(\frac{e_m^{(n)}}{n!}\bigg)^{k_n}-\frac{2e_m^{(n)}}{e_m^3}.
\end{align*}
From Lemma \ref{cor.r(1+r)} and assumption
$$\big|\partial_x^ne_m\big|\leq C^{n^2}\cdot n!\cdot  e_m\cdot \varphi_m^n,\quad m=1,2,$$
we have for $m=1,2$,
\beq\label{nderivativeem2}\Big|\partial^n\big(e_m^2\big)\Big|\leq 3 C^{n^2}\cdot n!\cdot e_m^2\cdot \varphi_m^n,
\quad
\Big|\partial^n\big(e_m^{-2}\big)\Big|\leq  3 C^{n^2}\cdot n!\cdot e_m^{-2}\cdot\varphi_m^n.\eeq
Moreover, we have
$$\partial^n(\sin^2\theta)=\sum_{\substack{k_1+2k_2+\cdots +nk_n=n\\ k=k_1+k_2+\cdots +k_n}} \frac{n!\cdot (-2)^{k-1}\cos(2\theta+\frac{k\pi}2)}{k_1!k_2!\cdots k_n!}\Bigg(\frac{\theta^{(1)}}{1!}\Bigg)^{k_1} \Bigg(\frac{\theta^{(2)}}{2!}\Bigg)^{k_2}\cdots \Bigg(\frac{\theta^{(n)}}{n!}\Bigg)^{k_n},$$
$$\partial^n(\cos^2\theta)=\sum_{\substack{k_1+2k_2+\cdots +nk_n=n\\ k=k_1+k_2+\cdots +k_n}} \frac{n!\cdot 2^{k-1}\cos(2\theta+\frac{k\pi}2)}{k_1!k_2!\cdots k_n!}\Bigg(\frac{\theta^{(1)}}{1!}\Bigg)^{k_1} \Bigg(\frac{\theta^{(2)}}{2!}\Bigg)^{k_2}\cdots \Bigg(\frac{\theta^{(n)}}{n!}\Bigg)^{k_n}.$$
By assumption, we have $|\theta^{(n)}|\leq C^{n^2}\cdot n!\cdot \mu_n^n$.
This gives that
\beq\label{nderivativesin2}
|\partial^n(\sin^2\theta)|\leq 3C^{n^2}\cdot n!\cdot \big(\max_{1\leq k\leq n} \mu_k\big)^n,\quad |\partial^n(\cos^2\theta)|\leq 3C^{n^2}\cdot n!\cdot\big(\max_{1\leq k\leq n} \mu_k\big)^n.
\eeq
Then we could obtain the estimates for the $n$th derivatives of $e_3$ by induction
$$|\partial^n e_3|\leq 4^n\cdot C^{n^2} \cdot n!\cdot e_3 \cdot \varphi_3^n.$$
In fact, assume that it holds for $1\leq n\leq k-1$, we are going to prove it for $n=k$. Take $n$th derivatives on each sides of (\ref{compe3}),
\begin{align} \nonumber
&\quad\quad  \partial^n(e_3^2)+\partial^n(e_3^{-2}) \\ \nonumber &=\sum_{n_1+n_2+n_3=n}\frac{n!}{n_1!n_2!n_3!}\partial^{n_1}(e_1^2)\cdot\partial^{n_2}(e_2^2)\cdot\partial^{n_3}(\sin^2\theta) \\ \nonumber &\quad\quad +\sum_{n_1+n_2+n_3=n}\frac{n!}{n_1!n_2!n_3!}\partial^{n_1}(e_1^2)\cdot\partial^{n_2}(e_2^{-2})\cdot\partial^{n_3}(\cos^2\theta) \\ \nonumber
&\quad\quad +\sum_{n_1+n_2+n_3=n}\frac{n!}{n_1!n_2!n_3!}\partial^{n_1}(e_1^{-2})\cdot\partial^{n_2}(e_2^2)\cdot\partial^{n_3}(\cos^2\theta)
\\ \label{nderivative} &\quad\quad +\sum_{n_1+n_2+n_3=n}\frac{n!}{n_1!n_2!n_3!}\partial^{n_1}(e_1^{-2})\cdot\partial^{n_2}(e_2^{-2})\cdot\partial^{n_3}(\sin^2\theta).
\end{align}
It is direct to see that the first term on the righthand side is the main term. By (\ref{nderivativeem2}) and (\ref{nderivativesin2}), we have
\begin{align}\nonumber
&\quad \quad \Big|\sum_{n_1+n_2+n_3=n}\frac{n!}{n_1!n_2!n_3!}\partial^{n_1}(e_1^2)\cdot\partial^{n_2}(e_2^2)\cdot\partial^{n_3}(\sin^2\theta) \Big|\\ \nonumber
&\leq \sum_{\substack{n_1+n_2+n_3=n\\ n_3\neq 0}}n!\cdot 2^{n_1}C^{n_1^2}\cdot e_1^2\cdot \varphi_1^{n_1} \cdot 2^{n_2}C^{n_2^2}\cdot e_2^2\cdot \varphi_2^{n_2} \cdot 3C^{n_3^2}\cdot \big(\max_{1\leq k\leq n_3} \mu_k\big)^{n_3}\\ \nonumber
&\hspace{6cm}+\sum_{n_1+n_2=n} n!\cdot 3C^{n_1^2}\cdot e_1^2\cdot \varphi_1^{n_1} \cdot 3C^{n_2^2}\cdot e_2^2\cdot \varphi_2^{n_2} \cdot \sin^2\theta\\ \nonumber
&\leq 4C^{n^2}\cdot n!\cdot e_1^2e_2^2\cdot\sum_{\substack{n_1+n_2+n_3=n\\ n_3\neq 0}} \varphi_1^{n_1}\varphi_2^{n_2} \big(\max_{1\leq k\leq n_3} \mu_k\big)^{n_3} \\ \label{nderivativeright} &\hspace{6cm}+3C^{n^2}\cdot n!\cdot e_1^2e_2^2\cdot\varphi_3^{n}\cdot \sin^2\theta.
\end{align}
By inductive assumption, we have
\begin{align}\label{nderivativeleft1}
\Bigg|\sum_{j=1}^{n-1} n!\cdot\frac{e_3^{(j)}}{j!}\cdot \frac{e_3^{(n-j)}}{(n-j)!}\Bigg|\leq 4^{n-1}\cdot C^{n^2}\cdot n!\cdot e_3^2\cdot \varphi_3^{n},
\end{align}
\beq\label{nderivativeleft2}
\Bigg|\sum_{\substack{k_1+2k_2+\cdots +nk_n=n\\ 2\leq k=k_1+k_2+\cdots +k_n}} \frac{n!(-1)^k (k+1)!}{k_1!k_2!\cdots k_n!e_m^{k+2}}\bigg(\frac{e_3^{(1)}}{1!}\bigg)^{k_1} \bigg(\frac{e_3^{(2)}}{2!}\bigg)^{k_1}\cdots \bigg(\frac{e_3^{(n)}}{n!}\bigg)^{k_n}\Bigg|\leq 4^{n-1}\cdot C^{n^2}\cdot n!\cdot e_3^{-2}\cdot \varphi_3^{n}.
\eeq
Also recall
$$e_3=e_1e_2\cdot |\sin \theta|+O(e_0^{-1})\gg 1.$$
Then by (\ref{nderivative}), (\ref{nderivativeright}), (\ref{nderivativeleft1}) and (\ref{nderivativeleft2}), we have
\begin{align*}&\quad \quad |e_3^{-1}\cdot \partial^n e_3|\\
&\leq 4C^{n^2} \cdot n! \sum_{\substack{n_1+n_2+n_3=n\\n_3\neq 0}}\varphi_1^{n_1}\varphi_2^{n_2} \big(\max_{1\leq k\leq n_3} \mu_k\big)^{n_3} |\sin \theta|^{-2}+3C^{n^2}\cdot n!\cdot\varphi_3^{n}+4^{n-1}\cdot C^{n^2}\cdot n!\cdot \varphi_3^{n}\\
&\leq 4C^{n^2} \cdot n! \sum_{\substack{n_1+n_2+n_3=n\\n_3\neq 0}}\varphi_1^{n_1}\varphi_2^{n_2} \big(|\sin \theta|^{-2}\cdot\max_{1\leq k\leq n} \mu_k\big)^{n_3}+3C^{n^2}\cdot n!\cdot \varphi_3^{n}+4^{n-1}\cdot C^{n^2}\cdot n!\cdot \varphi_3^{n}\\
&\leq 4^n\cdot C^{n^2} \cdot n!\cdot \varphi_3^n.
\end{align*}
Hence it holds that
$$|\partial^n e_3|\leq 4^n\cdot C^{n^2} \cdot n!\cdot e_3\cdot \varphi_3^n.$$

\

Next we consider the estimates for the angles. By (\ref{sformula}), the explicit expression of $s$ is given by
$$\frac\pi2-s(x)=\arctan \frac{\frac{w}{U}}{\sqrt{1+\frac{w^2}{U^2}}+1}.$$
By Fa\`a di Bruno's formula, we have
\begin{align*} \big|\partial^n(e_1^{-m})\big|
&\leq \Bigg|\sum_{\substack{k_1+2k_2+\cdots +nk_n=n\\ k=k_1+k_2+\cdots +k_n}} \frac{n!(-1)^k (m+k-1)!}{k_1!k_2!\cdots k_n!(m-1)!e_1^{m+k}}\bigg(\frac{e_1^{(1)}}{1!}\bigg)^{k_1} \bigg(\frac{e_1^{(2)}}{2!}\bigg)^{k_2}\cdots \bigg(\frac{e_1^{(n)}}{n!}\bigg)^{k_n} \Bigg|\\
&\leq C^{n^2}\cdot n!\cdot e_1^{-m}\cdot \varphi_1^n.
\end{align*}
Then by Leibniz formula, we have
\begin{align*}\left|\partial^n\left((1-e_1^{-4}e_2^{-4})\sin^2\theta +(e_2^{-4}-e_1^{-4})\cos^2\theta\right)\right|\leq C^{n^2}\cdot n!\cdot \tilde \varphi_3^n.\end{align*}
Again by Fa\`a di Bruno's formula, we have
\begin{align}\label{nderivativewU}
\left|\partial^n\Big((1-e_1^{-4}e_2^{-4})\sin^2\theta +(e_2^{-4}-e_1^{-4})\cos^2\theta\Big)^{-1}\right|\leq C^{n^2}\cdot n!\cdot \tilde \varphi_3^n.\end{align}
Recall
\begin{align*}
\frac{w}{U}=e_1^{-2}\cdot (1-e_2^{-4})\sin2\theta \cdot\Big((1-e_1^{-4}e_2^{-4})\sin^2\theta +(e_2^{-4}-e_1^{-4})\cos^2\theta\Big)^{-1}.
\end{align*}
Then combining (\ref{nderivativeem2}), (\ref{nderivativesin2}) and (\ref{nderivativewU}), we could apply Leibniz formula and Fa\`a di Bruno's formula to $\frac{w}{U}$ and then obtain
$$ \left|\partial^n\bigg(\frac{w}{U}\bigg)\right|\leq C^{n^2}\cdot n!\cdot e_1^{-2}\cdot \tilde \varphi_3^n.$$
Then it follows that
$$\left|\partial^n\bigg(\frac{w^2}{U^2}\bigg)\right|=\left|\sum_{j=0}^{n} n!\cdot\frac{(w/U)^{(j)}}{j!}\cdot \frac{(w/U)^{(n-j)}}{(n-j)!}\right|\leq 3 C^{n^2}\cdot n!\cdot e_1^{-4}\cdot \tilde \varphi_3^n.$$
Again by Fa\`a di Bruno's formula,
\begin{align*}
&\quad \quad \left|\partial^n\sqrt{1+\frac{w^2}{U^2}}\right|\\
&= \Bigg|\sum_{\substack{k_1+2k_2+\cdots +nk_n=n\\ k=k_1+k_2+\cdots +k_n}} \frac{n!(-1)^k (2k-3)!!}{k_1!k_2!\cdots k_n!2^{k}}\cdot\big(\frac wU\big)^{1-2k}\bigg(\frac{(w^2/U^2)^{(1)}}{1!}\bigg)^{k_1} \cdots \bigg(\frac{(w^2/U^2)^{(n)}}{n!}\bigg)^{k_n} \Bigg|\\
&\leq 4^n\cdot C^{n^2}\cdot n!\cdot e_1^{-\frac32}\cdot \tilde \varphi_3^{n},\quad \quad n\geq 1.
\end{align*}
Let $f_1:=\frac1{\sqrt{1+w^2/U^2}+1}$. Then
$$|\partial^n f_1|\leq  4^n\cdot C^{n^2}\cdot n!\cdot e_1^{-\frac32}\cdot \tilde \varphi_3^{n},\quad \quad n\geq 1.$$
It follows from Leibniz formula that
$$\left|\partial^n \Big(\frac{w}{U}\cdot f_1\Big)\right|=\left|\sum_{j=0}^{n} n!\cdot\frac{(w/U)^{(j)}}{j!}\cdot \frac{f_1^{(n-j)}}{(n-j)!}\right|\leq  4^n\cdot C^{n^2}\cdot n!\cdot e_1^{-\frac32}\cdot \tilde \varphi_3^{n},\quad \quad n\geq 1.$$
On the other hand, let
$f_2(x):=\arctan x$.
We have for any $n\in\mathbb Z_+$,
$$f_2^{(n)}(x) =\frac{(-1)^{n-1} (n-1)!}{2i\cdot (x-i)^n}-\frac{(-1)^{n-1} (n-1)!}{2i\cdot (x+i)^n},$$
where $i$ denotes the image unit.
Hence it holds that
$$\left|f_2^{(n)}\Big(\frac{w}{U}\cdot f_1\Big)\right|\leq  n!,\quad n\geq 1.$$
Recall that $\frac\pi2-s(x)=f_2\big(\frac{w}{U}\cdot f_1\big)$.
Then from Fa\`a di Bruno's formula, we have
$$|\partial^n s|=\left|\partial^nf_2\Big(\frac{w}{U}\cdot f_1\Big)\right|\leq \Big|\partial^n(f_2\circ f_1)\left(\frac wU\right)\Big|\leq 8^n\cdot C^{n^2}\cdot n!\cdot e_1^{-\frac32}\cdot \tilde \varphi_3^{n}.$$
This implies that
$$\|u_3-u_2\|_{C^k}\leq 8^k\cdot C^{k^2}\cdot k!\cdot e_2^{-\frac32}\cdot \tilde \varphi_3^{k},\quad \|s_3-s_1\|_{C^k}\leq 8^k\cdot C^{k^2}\cdot k!\cdot e_1^{-\frac32}\cdot \tilde \varphi_3^{k}.$$
\end{proof}

\begin{proof2}
The condition of Lemma \ref{lemma.hderivative} is a bit different from Lemma \ref{lemma.hderivative1}. However, we could proceed the similar process to obtain the conclusion. Since the proof is nearly the same, we omit it here.
\end{proof2}

\section*{Acknowledgements}
The authors are grateful to Professor Lingrui Ge, Jing Wang and Yiqian Wang for useful discussions.
J. Liang was supported by NSFC grant (12101311). J. You was  partially supported by National Key R\&D Program of China (2020 YFA0713300) and Nankai Zhide Foundation and  NSFC grant (11871286).


\begin{thebibliography}{10}
\bibitem{ALSZ} A. Avila, Y. Last, M. Shamis and Q. Zhou, \textit{On the abominable properties of the almost Mathieu operator with well-approximated frequencies}, Duke Math. J. \textbf{173} (2024), 603-672.


\bibitem{Bo99} J. Bochi, \textit{Discontinuity of the Lyapunov exponent for non-hyperbolic cocycles}, unpublished (1999).

\bibitem{Bo02} J. Bochi, \textit{Genericity of zero Lyapunov exponents}, Ergodic Theory Dynam. Systems \textbf{22} (2002), 1667-1696.

\bibitem{BJ02} J. Bourgain and S. Jitomirskaya, \textit{Continuity of the Lyapunov exponent for quasiperiodic operators with analytic potential}, J. Stat. Phys. \textbf{108} (2000), 1203-1218.

\bibitem{CGYZ} H. Cheng, L. Ge, J. You and Q. Zhou, \textit{Global rigidity results for ultra differential quasi-periodic cocycles and its spectral applications}, Adv. Math. \textbf{409} (2022), 1-65.

\bibitem{F} A. Furman, \textit{On the multiplicative ergodic theorem for the uniquely ergodic systems}, Ann. Inst. Henri Poincar\'e \textbf{33} (1997), 797-815.

\bibitem{GWYZ} L. Ge, Y. Wang, J. You and X. Zhao, \textit{Transition space for the continuity of the Lyapunov exponent of quasiperiodic Schr\"odinger cocycles}, arXiv:2102.05175 (2021).

\bibitem{GS01} M. Goldstein and W. Schlag, \textit{
 H\"older continuity of the integrated density of states for quasi-periodic Schr\"odinger equations and averages of shifts of subharmonic functions}, Ann. of Math. \textbf{2} (2001), 155-203.


\bibitem{HZ} R. Han and S. Zhang, \textit{Large deviation estimates and H\"older regularity of the Lyapunov exponents for quasi-periodic Schr\"odinger cocycles}, Int. Math. Res. Not. {\bf 3} (2022), 1666-1713.

\bibitem{J} S. Jitomirskaya, \textit{One-dimensional quasi-periodic operators: global theorey, duality, and sharp analysis of small denominators}, Proc. Int. Cong. Math. \textbf{2} (2022), 1090-1120.

\bibitem{JKS} S. Jitomirskaya, D. Koslover and M. Schulteis, \textit{Continuity of the Lyapunov exponent for general analytic quasiperiodic cocycles}, Ergodic Theory Dynam. Systems \textbf{29} (2009), 1881-1905.

\bibitem{JM11} S. Jitomirskaya and C. Marx, \textit{Continuity of the Lyapunov exponent for analytic quasi-periodic cocycles with singularities}, J. Fixed Point Theory Appl. \textbf{10} (2011), 129-146.

\bibitem{JM12} S. Jitomirskaya and C. Marx, \textit{Analytic quasi-periodic cocycles with singularities and the Lyapunov exponent of extended Harper’s model}, Comm. Math. Phys. \textbf{316} (2012), 237-267.

\bibitem{K} S. Klein, \textit{Anderson localization for the discrete one-dimensional quasi-periodic Schr\"odinger operator with potential defined by a Gevrey-class function}, J. Funct. Anal. {\bf 218} (2005) 255-292.

\bibitem{LWY} J. Liang, Y. Wang and J. You, \textit{H\"older continuity of Lyapunov exponent for a family of smooth Schr\"odinger cocycles}, Ann. Henri Poincar\'e {\bf 25} (2024), 1399-1444.

\bibitem{WY13} Y. Wang and J. You, \textit{Examples of discontinuity of Lyapunov exponent in smooth quasi-periodic cocycles}, Duke Math. J. {\bf 162} (2013), 2363-2412.

\bibitem{WY18} Y. Wang and J. You, \textit{The set of smooth quasi-periodic Schr\"odinger cocycles with positive Lyapunov exponent is not open}, Comm. Math. Phys. {\bf 362} (2018), 801-826.

\bibitem{WZ15} Y. Wang and Z. Zhang, \textit{Uniform positivity and continuity of Lyapunov exponents for a class of $C^2$ quasiperiodic Schr\"odinger cocycles}, J. Funct. Anal. {\bf 268} (2015), 2525-2585.

\bibitem{YZ} J. You and S. Zhang,  \textit{H\"older continuity of the Lyapunov exponent for analytic quasiperiodic Schr\"odinger cocycle with weak Liouville frequency}, Ergodic Theory Dynam. Systems. {\bf 34} (2014), 1395-1408.

\bibitem{Y} L. S. Young, \textit{Lyapunov exponents for some quasi-periodic cocycles}, Ergodic Theory Dynam. Systems {\bf 17} (1997), 483-504.

\end{thebibliography}
\end{document}